\documentclass[12pt]{amsart}

\usepackage[T1]{fontenc}
\usepackage{amsmath,amssymb} % formulas e simbolos adicionais
\usepackage{tikz} %quando codigo de um gráfico foi passado para o tex
\usepackage{pgf} %quando o código de um gráfico foi passado para o tex
\usepackage{hyperref,pstricks,graphicx} % enlaçamento
\usepackage{amsthm}% teoremas
\usepackage{geometry} % margens
\usepackage{enumitem}
\usepackage{ulem}
\usepackage{wrapfig}
\usepackage{multicol}
\usepackage{subcaption}

\newtheorem{theorem}{Theorem}[section]

\newtheorem{lemma}[theorem]{Lemma}

\newtheorem{remark}[theorem]{Remark}
\newtheorem{definition}[theorem]{Definition}
\newtheorem{proposition}[theorem]{Proposition}
\newtheorem{example}[theorem]{Example}

\title[Bifurcation and hyperbolicity for a nonlocal problem]{Bifurcation and hyperbolicity for a nonlocal quasilinear parabolic problem}

\author[J.M.  Arrieta]{Jos\'e M. Arrieta}
\address{Departamento de An\'alisis y Matem\'atica Aplicada, Universidad Complutense de Madrid, 28040 Madrid, Spain and Instituto de Ciencias Matem\'aticas ICMAT-UAM-UC3M-UCM}
\email{arrieta@mat.ucm.es}
\thanks{JMA is partially supported by grants PID2019-103860GB-I00 and CEX2019-000904-S ``Severo Ochoa Programme for Centres of Excellence in R\&D'' both from MICINN, Spain. Also by ``Grupo de Investigaci\'on 920894 - CADEDIF'', UCM, Spain.}

\author[A. N. Carvalho]{Alexandre N. Carvalho}
\thanks{ANC is partially supported by Grants FAPESP 2020/14075-6  and CNPq 306213/2019-2}

\author[E. M. Moreira]{Estefani M. Moreira}
\address{Departamento de Matem\'{a}tica, Instituto de Ci\^{e}ncias Ma\-te\-m\'{a}\-ti\-cas e de Computa\c{c}\~{a}o, Universidade de S\~{a}o Paulo, Campus de S\~{a}o Carlos, Caixa Postal 668, S\~{a}o Carlos, 13560-970, Brazil}
\email{andcarva@icmc.usp.br}
\email{estefani@usp.br}
\thanks{EMM is supported by CAPES, FAPESP Grant 2018/00065-9 and FAPESP Grant 2020/00104-4}

\author[J. Valero]{Jos\'e Valero}
\address{Centro de Investigaci\'on Operativa, Universidad Miguel Hern\'andez de Elche, Avenida Universidad s/n, 03202 Elche, Spain}
\email{jvalero@umh.es}
\thanks{JV is partially supported by the Spanish Ministry of Science, Innovation and Universities, project PGC2018-096540-B-I00, by the Spanish Ministry of Science and Innovation, project PID2019-108654GB-I00, by the Junta de Andalucía and FEDER, project P18-FR-4509, and by the Generalitat Valenciana, project PROMETEO/2021/063}


\begin{document}
	
\begin{abstract}
In this article, we study a one-dimensional nonlocal quasilinear problem of the form $u_t=a(\Vert u_x\Vert^2)u_{xx}+\nu f(u)$, with Dirichlet boundary conditions on the interval $[0,\pi]$, where $0<m\leq a(s)\leq M$ for all $s\in \mathbb{R}^+$ and $f$ satisfies suitable conditions.  We give a complete characterization of the bifurcations and of the hyperbolicity of the corresponding equilibria. With respect to the bifurcations we extend the existing result when the function $a(\cdot)$ is non-decreasing to the case of general smooth nonlocal diffusion functions showing that bifurcations may be pitchfork or saddle-node, subcritical or supercritical. We also give a complete characterization of hyperbolicity specifying necessary and sufficient conditions for its presence or absence. We also explore some examples to exhibit the variety of possibilities that may occur, depending of the function $a$, as the parameter $\nu$ varies.
\end{abstract}
	
	\maketitle
	
\section{Introduction}

This paper is dedicated to the study of the stationary solutions of the following nonlocal quasilinear parabolic problem
\begin{equation}\label{eq_non-local}
	\left\{\begin{aligned}
		& u_t=a(\Vert u_x\Vert^2)u_{xx}+\nu f(u),\ x \in (0,\pi), \  t>0,\\
		& u(0,t)=u(\pi,t)=0, \ t\geq 0, \\
		& u(\cdot,0)=u_0(\cdot)\in H^1_0(0,\pi),
	\end{aligned}\right. 
\end{equation}
where $\nu >0$ is a parameter, $a:\mathbb{R^+}\to [m,M]\subset (0,+\infty)$ is a continuously differentiable function, $f \in C^2(\mathbb{R})$, with 
\begin{equation} \label{eq_prop_f}
f(0)=0, \ f'(0)=1,\ 	f''(u)u< 0, \ \forall\, u \neq 0,
	\hbox{ and }\ {\limsup_{|u|\rightarrow +\infty } \frac{f(u)}{u} <0}.
\end{equation}

Here $\|\cdot\|$ denotes the usual norm in $L^2(0,\pi)$.

%%%\bigskip
For the stationary solutions of \eqref{eq_non-local} we analyze the bifurcations of the equilibria and their hyperbolicity. 
This analysis has been carried out in \cite{CLLM,Estefani2} for the particular case when $a$ is increasing and $f$ is odd (see also \cite{CCM-RV21C-Inf} for the study of the bifurcation when $f$ is not necessarily odd).  These two conditions considerably simplifies the structure of the bifurcations. In that case bifurcations are only from zero and they are all supercritical pitchfork bifurcations just like the local case $a=$const. Here we prove hyperbolicity and identify the bifurcations in the general case when $a$ is not necessarily increasing and $f$ is not necessarily odd. We will see that, in this general case, besides the supercritical pitchfork bifurcations, subcritical pitchfork bifurcations and saddle-node (subcritical and supercritical) bifurcations may occur.

The proof of hyperbolicity presented here is rather simple compared with that of \cite{Estefani2}. Nonetheless, an important part of the analysis is dependent on the analysis done in \cite{Estefani2} to view the quasilinear nonlocal problem \eqref{eq_non-local} as a semilinear nonlocal problem. We briefly recall this analysis to take advantage of it.

Consider the auxiliary semilinear nonlocal parabolic problem
\begin{equation}\label{eq_nl_changed}
\left\{
\begin{aligned}
& w_\tau = w_{xx} + \frac{\nu f(w)}{a(\| w_x\|^2)},  \, x\in (0, \pi),\, \tau>0,\\
& w(0, \tau)=w(\pi, \tau)=0,  \ \  \tau\geq 0,\\
& w(\cdot,0)=u_0(\cdot)\in H^1_0(0,\pi).
\end{aligned}
\right.
\end{equation}
Proceeding as in  \cite{Estefani2,CLLM}, \eqref{eq_nl_changed} is locally well-posed and the solutions are jointly continuous with respect to time and initial conditions. Changing the time variable to $t=\int_0^\tau a(\|w_x(\cdot,\theta)\|^2)^{-1}d\theta$ we have that $u(x,t)=w(x,\tau)$ is the unique solution of \eqref{eq_non-local}. As a consequence, \eqref{eq_non-local} is globally well-posed. If we define $S(t):H^1_0(0,\pi)\to H^1_0(0,\pi)$ by $S(t)u_0=u(t,u_0)$, $t\geq 0$, where $u(\cdot ,u_0):\mathbb{R}^+\to H^1_0(0,\pi)$ is the solution of \eqref{eq_non-local}, then $\{S(t):t\geq 0\}$ is a semigroup that has a global attractor $\mathcal{A}$. We say that a continuous function $v: \mathbb{R} \mapsto H^1_0(0,\pi)$ is a global solution for the semigroup $\{S(t):t\geq 0\}$ if it satisfies $v(t+s)=S(t)v(s)$ for all $t\geq 0$ and for all $s\in \mathbb{R}$. The global attractor can be characterized in terms of the global solutions in the following way
$$
\mathcal{A}=\{v(0): v:\mathbb{R} \to H^1_0(0,\pi) \hbox{ is a global bounded solution for } \{S(t):t\geq 0\} \}.
$$

In addition, the semigroup  $\{S(t):t\geq 0\}$ is gradient with Lyapunov function given by
\begin{equation}\label{LyapunovFunction}
V(u) = \frac12 A(\|u_x\|^2) -\nu \int_0^\pi F(u(x)) dx,
\end{equation}
where $A(s)={\displaystyle\int_0^s} a(\theta)d\theta$ and $F(s)={\displaystyle\int_0^s} f(\theta)d\theta$. 
Denote by $\mathcal{E}$ the set of equilibria of \eqref{eq_non-local}, that is, the set of solutions of
\begin{equation}\label{sl:non-local}\left\{
\begin{aligned}
& a(\Vert \varphi_x\Vert^2)\varphi_{xx}+\nu f(\varphi) =0, \ x \in (0,\pi), \\
& \varphi(0)=\varphi(\pi)=0.
\end{aligned}\right.
\end{equation} 
Then, for each $u \in H^1_0(0,\pi)$, $S(t)u_0 \stackrel{t\to +\infty}{\longrightarrow} \mathcal{E}$ and
\begin{equation*}
\begin{aligned}
\mathcal{A}=W^u(\mathcal{E}) = \Big\{ u\in H^1_0(0,\pi): & \ \hbox{there exists a global solution } \xi:\mathbb{R}\to H^1_0(0,\pi)\\
& \hbox{satisfying } \xi(0)=u \hbox{ and } \inf_{\varphi\in \mathcal{E}}\|\xi(t)-\varphi\|_{H^1_0(0,\pi)}\stackrel{t\to +\infty}{\longrightarrow} 0\Big\}.
\end{aligned}
\end{equation*}
In particular if $\mathcal{E}$ is finite
\begin{equation}\label{Att_charac}
\mathcal{A} = \bigcup_{\phi\in \mathcal{E}} W^u(\phi).
\end{equation}
Additionally, for any $u_0\in H^1_0(0,\pi)$ there is a $\phi\in \mathcal{E}$ such that $S(t)u_0\stackrel{t\to +\infty}{\longrightarrow} \phi$ and, for any bounded global solution $v:\mathbb{R} \to H^1_0(0,\pi)$ there are $\phi_-,\phi_+\in \mathcal{E}$ such that
$$
\phi_{-}\stackrel{t\to -\infty}{\longleftarrow}\xi(t)\stackrel{t\to +\infty}{\longrightarrow}\phi_{+}.
$$

Let us recall the definitions of local stable and unstable manifolds and the notion of hyperbolicity (see \cite{Estefani2}) which applies to \eqref{eq_non-local}.

\begin{definition} Given a neighborhood ${\mathcal V}_\delta(\phi)=\{u\in H^1_0(0,\pi): \|u-\phi\|_{H^1_0(0,\pi)}<\delta\}$ of $\phi$, the \emph{local stable and unstable sets} of $\phi$ associated to ${\mathcal V}_\delta(\phi)$, are given by
	\begin{equation*}
	\begin{split}
	W^{s,\delta}_{loc}(\phi) = &\{u \in H^1_0(0,\pi): S(t) u \in {\mathcal V}_\delta \hbox{ for all } t \geq  0, \hbox{ and } S(t) u\stackrel{t\to+\infty}{\longrightarrow} \phi\}, \\
	W^{u,\delta}_{loc}(\phi) = &\{u \in  H^1_0(0,\pi): \hbox{ there exists a global solution } v  \hbox{  of }  \{S(t)\colon t\geq 0\}  \\
	& \hspace{40pt} \hbox{ with }v(0)=u, \ v(t)\in {\mathcal V}_\delta \hbox{ for all }  t\leq 0 \hbox{ and }v(t)\stackrel{t\to -\infty}{\longrightarrow} \phi\}.
	\end{split}
	\end{equation*}
\end{definition}

When $\phi$ is a maximal invariant set in a neighborhood of itself and $W^{u,\delta}_{loc}(\phi)=\{\phi\}$, it is asymptotically stable; otherwise it is unstable. In this case all solutions that remain in ${\mathcal V}_\delta(\phi)$ for all $t\geq 0$ ($t\leq 0$) must converge forwards (backwards) to $\phi$. We refer to this property as topological hyperbolicity.

\begin{definition}[Strict Hyperbolicity, \cite{Estefani2}]\label{SH} An equilibrium $\phi$ of \eqref{eq_non-local} is said to be {\bf hyperbolic} if there are closed linear subspaces $X_u$ and $X_s$ of $H^1_0(0,\pi)$ with $H^1_0(0,\pi)=X_u\oplus X_s$ such that
	\begin{itemize}
		\item $\{\phi\}$ is topologically hyperbolic.
		\item The local stable and unstable sets are given as graphs of Lipschitz maps $\theta_u:X_u\to X_s$ and 
		$\theta_s:X_s\to X_u$, with Lipschitz constants $L_s$, $L_u$ in $(0,1)$ and such that $\theta_u(0)=\theta_s(0)=0$, and there exists $\delta_0>0$ such that, given $0<\delta<\delta_0,$ there are $0<\delta''<\delta'<\delta$ such that 
		\begin{equation*}
		\begin{split}
		\{\phi\!+\!(x_u,\theta_u(x_u))\!:\! x_u\!\in\! X_u, \|x_u\|_{H^1_0}\!<\!\delta''\}&\!\subset\! W^{u,\delta'}_{loc}(\phi) \\
		&\!\subset\! \{\phi\!+\!(x_u,\theta_u(x_u))\!:\! x_u\!\in\! X_u, \|x_u\|_{H^1_0}\!<\!\delta\},\\
		\{\phi\!+\!(\theta_s(x_s),x_s)\!:\! x_s\in X_s, \|x_s\|_{H^1_0}\!<\!\delta''\}&\!\subset\! W^{s,\delta'}_{loc}(\phi)\\
		&\!\subset\! \{\phi\!+\!(\theta_s(x_s),x_s)\!:\! x_s\!\in\! X_s, \|x_s\|_{H^1_0}\!<\!\delta\}.
		\end{split}
		\end{equation*}
	\end{itemize}
\end{definition}

Proceeding as in \cite{Estefani2}, we will show strict hyperbolicity of equilibria for \eqref{eq_non-local} in the following way. First we note that the equilibria of  \eqref{eq_non-local} and \eqref{eq_nl_changed} are the same. Then we consider the linearization around an equilibrium for the semilinear problem \eqref{eq_nl_changed} and prove their hyperbolicity (showing that zero is not in the spectrum of the linearized self-adjoint nonlocal operator). Then we use the solution dependent change of time scale to conclude the hyperbolicity for \eqref{eq_non-local}.

In \cite{CCM-RV21C-Inf,CLLM,Estefani2}, the authors proved the following:
\begin{theorem}\label{theo:exist.equilibria} Assume that the function $a(\cdot)$ is increasing
and that $f$ is odd. 
	If $a(0)N^2< \nu\leq a(0)(N+1)^2,$ then there are $2N+1$ equilibria of the equation \eqref{eq_non-local}; $\{0\}\cup\left\{\phi_{j}^\pm: j=1,\dots, N\right\}$,
	where $\phi^+_{j}$ and $\phi^-_{j}$  have $j+1$ zeros in $[0,\pi]$ and $\phi^-_{j}(x)=-\phi^+_{j}(x)$ for all $x\in[0,\pi]$ and $\phi^+_{j}(x)>0$ for all $x\in(0,\frac{\pi}{j})$.
	The sequence of bifurcation given above satisfies:
	\begin{itemize}
		\item[] {\underline{Stability}}: If $\nu\leq a(0)$, $0$ is the only equilibrium of \eqref{eq:C-I} and it is stable. If $\nu>a(0)$, the positive equilibrium $\phi_1^+$ and the negative equilibrium $\phi_1^-$ are stable and any other equilibrium is unstable.

		\item[] {\underline{Hyperbolicity}}:  For all $\nu>0$, the equilibria are hyperbolic with the exception of $0$ in the cases $\nu=a(0)N^2$,  for $N \in \mathbb{N}^*$.

	\end{itemize}
\end{theorem}

It is well known (see \cite{C-I74}) that, for each $\nu\in (N^2,\infty)$ the problem \eqref{eq_non-local} with $a\equiv 1$ admits exactly two equilibria $\phi_{N,\nu}^\pm$ that vanish exactly $N-1$ times in $(0,\pi)$ and such that $(\pm 1) (\phi_{N,\nu}^\pm)'(0)>0$.  Consequently, the problem \eqref{eq_non-local} with $a\equiv 1$ admits $2N+1$ equilibria, given by $\{0\}\cup\{\phi_{j,\nu}^+,\phi_{j,\nu}^- : j=1,\dots, N\}$, where
\begin{itemize}
	\item $\phi_{j,\nu}^\pm$ have $j-1$ zeros in $(0,\pi)$, $j=1,\dots, n$.
\end{itemize}

Our main result in this paper is inspired in the work of \cite{C-I74} and aims to show that the nonlocal diffusion brings many new interesting features to the bifurcation problem. It can be stated as follows.

\begin{theorem}\label{theo:Identify_equil_nlA}
 
For each positive integer $j$ there are  two continuous, strictly decreasing functions $c_j^+, c_j^-:\mathbb{R}^+\to (0,\frac{1}{j^2}]$ with $c_j^\pm(r)\stackrel{r\to +\infty}{\longrightarrow} 0$ such that:

For $\nu>0$ and $r>0$, \eqref{eq_non-local} has an equilibrium $\phi$, with $j-1$ zeros in the interval $(0,\pi)$, such that $\phi_x(0)>0$ (resp. $\phi_x(0)<0$) and  $\|\phi_x\|^2=r$ if and only if $\nu c_j^+(r)=a(r)$  (resp. $\nu c_j^-(r)=a(r)$). Furthermore, the equilibrium is hyperbolic if and only if, $a'(r)\ne\nu (c_j^+)'(r)$ (resp. $a'(r)\ne\nu (c_j^-)'(r)$).
\end{theorem}

\begin{remark}
Note that, Theorem \ref{theo:Identify_equil_nlA} characterizes all equilibria of \eqref{eq_non-local}. Also, it is only required for the function $a:\mathbb{R^+}\to [m,M]\subset (0,+\infty)$ to be continuously differentiable, that is, $a(\cdot)$ is not necessarily increasing.

The study of existence of equilibria requires only the continuity of $a(\cdot)$. The differentiability of $a(\cdot)$ is used to analyze the behavior near the equilibria.
\end{remark}

Assuming only that $f$ and $a$ are continuous, $a(s)\geq m$, for all $s\in \mathbb{R}^+$ and that ${\displaystyle\limsup_{|s|\to +\infty}} \frac{f(s)}{s}=\beta<+\infty$ it has been proved in \cite{CCM-RV21C-Inf} that there exist a solution for \eqref{eq_non-local}, defined for all $t\geq 0$, for each $u_0(\cdot)\in H^1_0(0,\pi)$ and that \eqref{eq_non-local} defines a multivalued semiflow. In addition if $a$ is either non-decreasing or bounded above and $f$ satisfies some growth and dissipativity conditions the authors show that the multivalued semiflow has a global attractor which is characterized as the unstable set of the equilibria. Under some additional assumptions it is also proved in \cite{CCM-RV21C-Inf} that the set of equilibria has  at least $2N+1$ points if $\lambda > a(0)N^2$ and exactly $2N+1$ if a is non-decreasing and $a(0)(N+1)^2\geqslant \lambda > a(0)N^2$. In this paper our focus is on the bifurcation, stability and hyperbolicity of equilibria assuming that $a$ and $f$ are smooth.

In this paper, we pay attention to the case where the function $a(\cdot)$ is not necessarily increasing.  Observe that the diffusion coefficient $a(\|u_x\|^2)$ in \eqref{eq_non-local} depends on the $L^2$ norm of the gradient of the solution.  This means that, roughly speaking, if the function $a(\cdot)$ is increasing, then states with large gradients will have large diffusion coefficient and in some sense, the diffusion mechanism is more efficient in trying to smooth out the solution and definitely in stabilizing the system. Therefore, the dynamics, at least in terms of stability of positive equilibria is expected to be similar to the classical case in which the diffusion does not depend on the state, see \cite{Estefani2}.  On the other hand, if the function $a(\cdot)$ is decreasing for some range of the parameter, it is possible that the systems favors states with large gradients and it may destabilize the system.  This is what actually may occur and, as we will see,  we may have situations in which some not changing sign equilibria  may become unstable, see Theorem \ref{theo:Morse_index} below.

%%%\par\medskip 

This paper is organized as follows. In Section \ref{CI-properties} we study fine properties of the solutions of the Chafee-Infante model \eqref{eq_non-local} with $a\equiv 1$. In Section \ref{Identification} we explain how solutions of \eqref{eq_non-local} can be retrieved from the solutions of \eqref{eq_non-local} with $a\equiv 1$. In Section \ref{hyperbolicity} we give a full characterization of the bifurcations as a function of the parameter $\lambda$ and of the function $a$ and also characterize the exact points where we may loose hyperbolicity of the equilibria. Finally in Section \ref{examples} we show some examples to exhibit the variety of behaviors one may identify for different functions $a$.

	\section{Properties of equilibria for the Chafee-Infante model}\label{CI-properties}
	
	Consider the problem
	\begin{equation}\label{eq:C-I}
		\left\{\begin{aligned}
			& u_t=u_{xx}+\lambda f(u),\ x \in (0,\pi), \  t>0,\\
			& u(0,t)=u(\pi,t)=0, \ t\geq 0, \\
			& u(\cdot,0)=u_0(\cdot)\in H^1_0(0,\pi),
		\end{aligned}\right.
	\end{equation}
where $\lambda >0$ is a parameter, $f \in C^2(\mathbb{R})$ satisfying \eqref{eq_prop_f}.
% {\color{blue} %satisfying \eqref{eq_prop_f}.}
%	 $f(0)=0$, $f'(0)=1$,
%	\begin{equation} \label{eq_prop_fB}
%		f''(u)u < 0, \ \forall\, u \neq 0,\
%		\hbox{ and }\ {\limsup_{|u|\rightarrow +\infty } \frac{f(u)}{u} <0.}
%	\end{equation}
%}

This problem is known as the Chafee-Infante problem and is a very well-studied nonlinear dynamical system. In fact, we can say that it is the best understood example in the literature referring to the characterization of a non-trivial attractor of an infinite dimensional problem. Chafee and Infante started the description of the attractor in \cite{C-I-n2,C-I74} by showing that the problem admits only a finite number of equilibria which bifurcate from zero as the parameter $\lambda>0$ increases. Also, these equilibria are all hyperbolic, with the exception of the zero equilibrium for $\lambda = N^2$, for $N \in \mathbb{N}$.

\begin{remark}
The dissipativity condition \eqref{eq_prop_f} can be relaxed (with very little changes) to include the possibility that inequality is not strict. We chose to keep the analysis as simple as possible.
\end{remark}

	\begin{theorem}\label{theo:C-I_equilibria}
		
For each $\lambda\in (N^2,+\infty)$, $N \in \mathbb{N}$, problem \eqref{eq:C-I} admits exactly two equilibria  $\phi_{N,\lambda}^+$ and $\phi_{N,\lambda}^-$ that vanish exactly $N-1$ times in the interval $(0,\pi)$ and such that  $(\phi_{N,\lambda}^{+})'(0)>0$ and $(\phi_{N,\lambda}^{-})'(0)<0$.  Hence, if $\lambda\in (N^2,(N+1)^2]$, then \eqref{eq:C-I} admits exactly the following $2N+1$ equilibria:  $\{0\}\cup\{\phi_{j,\lambda}^+,\phi_{j,\lambda}^- : j=1,\dots, N\}$.
		
For each $1\leq j\leq N$, the linear operator $L_j^{\lambda,\pm}:H^2(0,\pi)\cap H^1_0(0,\pi)\subset L^2(0,\pi)\to L^2(0,\pi)$ defined by $L_j^{\lambda,\pm}u=u_{xx} + \lambda f'(\phi_{j,\lambda}^\pm)u$, $u\in H^2(0,\pi)\cap H^1_0(0,\pi)$ is a self-adjoint unbounded operator with compact resolvent. All eigenvalues of $L_j^{\lambda,\pm}$ are simple, zero is not an eigenvalue and exactly $j-1$ eigenvalues are positive,  $j=1,\dots, N$.
	\end{theorem}

If $j$ is a positive integer, many properties of the equilibrium $\phi_{j,\lambda}^\pm$ of \eqref{eq:C-I} are proved using the properties of the time maps which we briefly recall for later use. 

%%%%\bigskip

Since, for $\lambda \in (j^2,+\infty)$, $\phi_{j,\lambda}^\pm$ are the solutions of \eqref{eq:C-I} with $j-1$ zeros in the interval $(0,\pi)$ and $i(\phi_{j,\lambda}^i)'(0)>0$, $i\in\{+,-\}$, they are solutions of the initial value problem
\begin{equation}\label{eq:ci-ivp}
\begin{split}
&u_{xx}+\lambda f(u)=0,\ x >0,\\
&u(0)=0, \ u'(0)=v_0,
\end{split}
\end{equation}
where $v_0>0$ is suitably chosen in such a way that $u(\pi)=0$. For a given $v_0$ let $\lambda E=\frac{v_0^2}{2}\in [0,\min\{F(z^+),F(z^-)\}]$, where $F(u)=\int_0^uf(s)ds$, $z^+$ (resp. $z^-$) is the positive (resp. negative) zero of $f$, and note that a solution of \eqref{eq:ci-ivp} must satisfy
$$
\frac{u'(x)^2}{2} +\lambda F(u) = \lambda E.
$$
Let $U^+(E)>0$ and $U^-(E)<0$ be defined as the unique numbers in $[0,z^+]$ and $[z^-,0]$, respectively, with  $F(U^\pm(E))=E$. Then, if 
\begin{equation}\label{time_map}
\tau_\lambda^i(E)=i \left(\frac{2}{\lambda}\right)^\frac12\int_0^{U^i(E)} (E-F(u))^{-\frac12} du, \ i\in\{+,-\},
\end{equation} 
we have, for $j$ odd,
$$
{\mathcal{T}}_\lambda^+(E)=\frac{j+1}{2} \tau_\lambda^+(E) + \frac{j-1}{2} \tau_\lambda^-(E), \quad {\mathcal{T}}_\lambda^-(E)=\frac{j+1}{2} \tau_\lambda^-(E) + \frac{j-1}{2} \tau_\lambda^+(E) 
$$
or, for $j$ even,
$$
{\mathcal{T}}_\lambda^\pm(E)=\frac{j}{2} \tau_\lambda^+(E) + \frac{j}{2} \tau_\lambda^-(E).
$$
The choices of $E$ that gives us the solutions $\phi_{j,\lambda}^+$ are ${\mathcal{T}}_\lambda^+(E_{j,\lambda}^+)=\pi$.

For completeness we give a simple proof that the equilibria of the Chafee-Infante equation \eqref{eq:C-I} are all hyperbolic (see \cite[Section 24F]{Smoller}) with the only exception being the equilibrium $\phi_0\equiv 0$ and exactly when $\lambda=N^2$, $N$ a positive integer. This shows, in particular, that bifurcations only occur from the $\phi_0$.

We prove only the hyperbolicity of $\phi_{j,\lambda}^+$, the other case is similar. We consider the family $u(\cdot,E)$ of solutions of the problem
\begin{equation}\label{edo_aux}
\begin{split}
&u''(x) + \lambda f(u(x))= 0, \\
&u(0,E)=0,\ u'(0,E)=\sqrt{2\lambda E}\ \hbox{ and }\ u(\tau^+_\lambda (E))=0.
\end{split}
\end{equation}
Consequently, $\eta=(\phi_{j,\lambda}^+)_x$ and $\psi=\frac{\partial u}{\partial E} (x,E)\big|_{E=E_{j}^+(\lambda)}$ are solutions of 
\begin{equation}\label{edo4}
v''(x) + \lambda f'(\phi_{j,\lambda}^+) v(x)= 0
\end{equation}
with $\eta(0)\neq 0$, $\eta'(0)=0$ and $\psi(0)=0$, $\psi'(0)=\frac{\sqrt{\lambda}}{\sqrt{2 E_{j}^+(\lambda)}}\neq 0$. This proves that $\eta$ and $\psi$ are linearly independent and any solution of \eqref{edo4} must be of the form
$$
\omega=c_1\eta+c_2\psi, \ \mbox{ for } c_1,c_2 \in \mathbb{R}.
$$
Let us show that if $\omega(0)=\omega({\mathcal{T}}_\lambda^+(E_{j}^+(\lambda)))=0$ then, necessarily, $w\equiv 0$.
In fact, $\psi(0)=0$, $\eta(0)\neq 0$ and $c_1\eta(0)+c_2\psi(0)=0$ implies $c_1=0$. Now, since $u({\mathcal{T}}_\lambda^+(E),E)=0$ for all $E$, we have that $0=\frac{\partial u}{\partial x}({\mathcal{T}}_\lambda^+(E),E)({\mathcal{T}}_\lambda^+(E))'(E)+\frac{\partial u}{\partial E}({\mathcal{T}}_\lambda^+(E),E)$. It is clear that $\frac{\partial u}{\partial x}({\mathcal{T}}_\lambda^+(E),E)\neq 0$ and since that $({\mathcal{T}}_\lambda^+(E))'(E)\neq 0$ (see \cite{C-I74}), we have that $\psi({\mathcal{T}}_\lambda^+(E_{j}^+(\lambda)))=\frac{\partial u}{\partial E}({\mathcal{T}}_\lambda^+(E_{j}^+(\lambda),E_{j}^+(\lambda))\neq 0$. Hence, we also have that $c_2=0$ and the only solution $\omega$ of \eqref{edo4} which satisfies $\omega(0)=\omega(\pi)=0$ is $\omega\equiv 0$. This proves that $0$ is not in the spectrum of the linearization around $\phi$.

%%%%\bigskip	

Now we study the properties of the functions $(j^2,+\infty)\ni \lambda\mapsto \phi_{j,\lambda}^\pm\in H^1_0(0,\pi)$, $j=1,2,3\cdots$.
		
\begin{theorem}\label{prop-aux-func}
For each positive integer $j$, the two functions $(j^2,+\infty)\ni \lambda\mapsto \phi_{j,\lambda}^+\in H^1_0(0,\pi)$, $(j^2,+\infty)\ni \lambda\mapsto \phi_{j,\lambda}^-\in H^1_0(0,\pi)$ are continuously differentiable and consequently the two functions $(j^2,+\infty)\ni \lambda\mapsto \|(\phi_{j,\lambda}^\pm)_x\|^2\in (0,+\infty)$ are strictly increasing, continuously differentiable and $\|(\phi_{j,\lambda}^\pm)_x\|^2\stackrel{\lambda\to+\infty}{\longrightarrow}+\infty$.
\end{theorem}

	\begin{proof} Consider $i \in \{+,-\}$. To show that $(j^2,+\infty)\ni \lambda\mapsto \phi_{j,\lambda}^i\in H^1_0(0,\pi)$ is continuously differentiable at a point $\lambda_0$ we recall that, for each $\lambda \in (j^2,+\infty)$ we already know that $\phi_{j,\lambda}^i$ is hyperbolic. Hence, to obtain the differentiability at $\lambda=\lambda_0\in (j^2,+\infty)$ we recall that, for $\lambda$ near $\lambda_0$,  $\phi_{j,\lambda}^i=\phi_{j,\lambda_0}^i+v$, where $v$ is the only fixed point of the map
		$$
		T_{j,\lambda}^i v := -\phi_{j,\lambda_0}^i - (L_j^{\lambda_0,i})^{-1} \left( \lambda f(v+\phi_{j,\lambda_0}^i) - \lambda_0 f'(\phi_{j,\lambda_0}^i)v - \lambda_0 f'(\phi_{j,\lambda_0}^i) \phi_{j,\lambda_0}^i\right)
		$$
		in a small neighborhood of zero in $H^1_0(0,\pi)$. Now, since $(j^2,+\infty)\ni \lambda\mapsto T_{j,\lambda}^i\in \mathcal C(H^1_0(0,\pi))$ is continuously differentiable we have that $(j^2,+\infty)\ni \lambda\mapsto \phi_{j,\lambda}^i\in H^1_0(0,\pi)$ is continuously differentiable and the result follows.

The proof that $(j^2,+\infty)\ni \lambda\mapsto \|(\phi_{j,\lambda}^i)_x\|^2\in (0,+\infty)$ is strictly increasing and that $\|(\phi_{j,\lambda}^i)_x\|^2\stackrel{\lambda\to +\infty}{\longrightarrow}+\infty$ follows from the results in \cite[Lemma 5]{CCM-RV21C-Inf} and of the analysis done next.

It has been shown in \cite{C-I74} that the time maps $\tau_\lambda^\pm(\cdot)$, defined in \eqref{time_map}, are strictly increasing functions. Also, for a fixed $E$, clearly $\lambda\mapsto\tau_\lambda^\pm(E)$ is strictly decreasing. Hence, since ${\mathcal{T}}_\lambda^+(E_{j,\lambda}^+)=\pi$, we must have that $i U^i(E_j^i(\lambda))$, $i \in \{+,-\}$, is strictly increasing.

It follows that
$$
g(\lambda):=\int_0^{\tau_\lambda^+(E_j^\pm(\lambda))} ((\phi_{j,\lambda}^\pm)_x)^2dx =  \sqrt{2\lambda}\int_0^{U^+(E_j^\pm(\lambda))} \sqrt{E_j^\pm(\lambda)-F(v)}dv
$$ 
and
$$\int_0^{U^+(E_j^\pm(\lambda))} \sqrt{E_j^\pm(\lambda)-F(v)}dv$$
is an strictly increasing function of $\lambda$. Consequently, $g(\lambda)\stackrel{\lambda\to +\infty}{\longrightarrow} +\infty$ and we must have that $\|(\phi_{j,\lambda}^i)_x\|^2\stackrel{\lambda\to+\infty}{\longrightarrow}+\infty$, completing the proof. 
\end{proof}

Let us consider an alternative simple direct proof of this theorem without using the differentiability results for fixed points. First we show that $(j^2,+\infty)\ni \lambda\mapsto \phi_{j,\lambda}^i\in H^1_0(0,\pi)$ is continuous in $H^1_0(0,\pi)$ $($or $C^1(0,\pi))$. For simplicity of notation we will write $\phi_{\lambda}$ for $\phi_{j,\lambda}^i$.

Let us to show that if $\lambda_n \rightarrow \lambda_0\in (j^2,+\infty)$, we must have that $\|\phi_{\lambda_n}-\phi_{\lambda_0}\|_{H^1_0(0,\pi)}\rightarrow 0$. Since
\begin{equation}\label{equation}
(\phi_{\lambda})_{xx}(r)+\lambda f(\phi_{\lambda})=0
\end{equation}
and the dissipativity condition in \eqref{eq_prop_f} 
we have that there is a constant $M>0$ such that
$$
{\int_0^\pi ((\phi_{\lambda})_x)^2dx=\lambda\int_0^\pi f(\phi_{\lambda})\phi_{\lambda}dx\leq \lambda M.}
$$
	
Therefore, the family $(j^2,+\infty)\ni\lambda\mapsto \phi_{\lambda}\in H^1_0(0,\pi)$ is bounded in bounded subsets of $(j^2,+\infty)$. Since $H^1_0(0,\pi)\hookrightarrow C([0,\pi])$ it is also uniformly bounded in $C([0,\pi])$ uniformly in bounded subsets of $(j^2,+\infty)$. From the continuity of $f$, the same is true for $(j^2,+\infty)\ni\lambda\mapsto f\circ \phi_{\lambda}\in C([0,\pi])$ and, using \eqref{equation}, for $(j^2,+\infty)\ni\lambda\mapsto (\phi_{\lambda})_{xx}\in C([0,\pi])$.
	
	It follows from the compact embedding of $H^2(0,\pi)$ into $H^1_0(0,\pi)$ that there is a subsequence $\{\lambda_{n_k}\}$  of $\{\lambda_n\}_{n\in \mathbb{N}}$ such that $\phi_{\lambda_{n_k}}\stackrel{k\to +\infty}{\longrightarrow} w$ in $H^1_0(0,\pi)$. Now, since
	$$
	\int_0^\pi (\phi_{\lambda_{n_k}})_x v_xdx=\lambda_n\int_0^\pi f(\phi_{\lambda_{n_k}})vdx,
	$$
	for all $v \in H^1_0(0,\pi)$, passing to the limit as $k\to +\infty$ we have that
	$$
	\int_0^\pi w_x v_xdx=\lambda_0\int_0^\pi f(w)vdx,
	$$
	and $w$ is a weak solution of \eqref{equation}. Hence, since $w$ also converges in the $C^1(0,\pi)$ norm, $w\equiv 0$ or $w=\phi_{\lambda_0}$. To see that $w\not\equiv 0$ we recall that $(j^2,+\infty)\ni\lambda\mapsto \int_0^\pi ((\phi_{\lambda})_x)^2$ is an strictly increasing function of $\lambda$. This shows the continuity of the function $(j^2,+\infty)\ni\lambda\mapsto \phi_{\lambda}\in H^1_0(0,\pi)$.

	Let us now prove that $(j^2,+\infty)\ni \lambda\mapsto \phi_\lambda\in H^1_0(0,\pi) \hbox{ or } C^1_0(0,\pi)$ is continuously differentiable. Fix $\lambda\in (j^2,+\infty)$ and consider $\delta>0$ is such that $\lambda+h\in (j^2,+\infty)$ for all $h\in (-\delta,\delta)$. Denote $w(h)=\frac{\phi_{\lambda+h}-\phi_\lambda}{h}$.
	
	Now

$$
	w(h)_{xx}+f(\phi_{\lambda+h})+\lambda\left(\frac{f(\phi_{\lambda+h})-f(\phi_{\lambda})}{h}\right)
$$
$$
	=w(h)_{xx}+{f(\phi_{\lambda+h})}+\lambda f'(\theta \phi_{\lambda+h}+(1-\theta)\phi_{\lambda})w(h)=0
	$$
	
	and
	$$
	\|w(h)_{x}\|^2\leq C\|w(h)\|^2+ C.
	$$
	Hence, proceeding as before we show that $\{w(h): h \in (0,1]\}$ is uniformly bounded in $H^1_0(0,\pi)$ and so it is in $L^2(0,\pi)$ and $C(0,\pi)$.
	%%%%%
	
	Hence, using that $f \in C^2(\mathbb{R})$, $h \in (0,\delta]$, we must have that 
	$$
	\sup_{h\in (0,\delta]}\sup_{y\in [0,\pi]}|w(h)_{xx}(y)|<+\infty.
	$$
	
	Therefore, the sequence $\{w(h): h \in (0,1]\}$ is uniformly bounded in $H^2(0,\pi)$. Hence, we may assume that $w(h)\rightarrow \bar{w}$ in $H^1_0(0,\pi)$ (so as $C^1[0,\pi]$) as $h\rightarrow 0$.
	
	Since, for all $v \in H^1_0(0,\pi)$, $h \in (0,\delta)$, we have
	$$
	\int_0^\pi w(h)_{x}v_x=\int_0^\pi f(\phi_{\lambda+h})v+\lambda\int_0^\pi \lambda f'(\theta \phi_{\lambda+h}+(1-\theta)\phi_{\lambda})w(h)v
	$$
	we find, using the continuity of the function $(j^2,+\infty)\ni \lambda \mapsto \phi_\lambda\in H^1_0(0,\pi)$, that
	\begin{equation}\label{eq:bar_w}
		-\left<\bar{w}_x, v_x\right>+\lambda\left< f'(\phi_\lambda)\bar{w}, v\right>+{\left<f(\phi_\lambda), v\right>}=0,
	\end{equation}
	for all $v \in H^1_0(0,\pi)$. That is $\bar{w}$ is the only solution of
	\begin{equation*}
		\left\{
		\begin{aligned}
			&u_{xx}+\lambda f'(\phi_\lambda)u+{f(\phi_\lambda)}=0\\
			& u(0)=u(\pi)=0.
		\end{aligned}
		\right.
	\end{equation*}
	From this we have the differentiability of the function $(j^2,+\infty)\ni \lambda \mapsto \phi_\lambda\in H^1_0(0,\pi)$.
	
	\begin{remark}
		The same reasoning can be used to show that $(j^2,+\infty)\ni \lambda \mapsto \phi_\lambda\in H^1_0(0,\pi)$ is twice continuously differentiable.
	\end{remark}
	
%%%%	\bigskip
	
	Let us now define an auxiliary function which will allow us to see the equilibria of \eqref{eq:C-I} as equilibria of a nonlocal problem. As we have seen in Theorem \ref{prop-aux-func}, for each positive integer $j$,  the two functions $[j^2,+\infty)\ni \lambda \mapsto \phi_{j,\lambda}^\pm\in H^1_0(0,\pi)$ are continuously differentiable and $[j^2,+\infty)\ni \lambda \mapsto \int_0^\pi((\phi_{j,\lambda}^\pm)_x)^2dx\in (0,+\infty)$ is strictly increasing (see \cite{CCM-RV21C-Inf}) and continuously differentiable. Since $\phi_{j,j^2}^\pm=0$ \cite{C-I74}, we also know that $\int_0^\pi((\phi_{j,\lambda}^\pm)_x)^2dx \rightarrow 0$ as $\lambda\rightarrow j^2$.

\begin{definition}\label{def_cjpm}
	For each positive integer $j$ and $r\geq 0$, let $\lambda_{j,r}^\pm\in [j^2,+\infty)$ be the unique $\lambda$ such that $\int_0^\pi((\phi_{j,\lambda}^\pm\!)_x)^2=r$. Let $c_{j}^\pm:[0,+\infty) \to (0, \frac{1}{j^2}]$ be the function defined by $c_j^\pm(r)= \frac{1}{\lambda_{j,r}^\pm}$, for each $r\geq 0$. 
\end{definition}

Clearly the two functions $c_j^\pm(\cdot)$ are strictly decreasing and continuously differentiable with $\lim_{r\to 0}c_j^\pm(r)={\displaystyle\frac{1}{j^2}}$.  With the aid of this auxiliary function we can rewrite the problem
	\begin{equation}\label{eq_ci}
		\left\{
		\begin{aligned}
			&(\phi_{j,\lambda}^\pm)_{xx}+\lambda f(\phi_{j,\lambda}^\pm)=0,\\
			& \phi_{j,\lambda}^\pm(0)=\phi_{j,\lambda}^\pm(\pi)=0,
		\end{aligned}
		\right.
	\end{equation}
	as the following `nonlocal' problem
	\begin{equation}\label{local_nonlocal}
		\left\{
		\begin{aligned}
			&c_j^i(\|(\phi_{j,\lambda}^\pm)_x\|^2)(\phi_{j,\lambda}^\pm)_{xx}+f(\phi_{j,\lambda}^\pm)=0,\\
			& \phi_{j,\lambda}^\pm(0)=\phi_{j,\lambda}^\pm(\pi)=0.
		\end{aligned}
		\right.
	\end{equation}

For simplicity of notation we will write $c(\cdot)$ to denote one of the two functions $c_{j}^\pm(\cdot)$ and $\phi_\lambda$ for $\phi_{j,\lambda}^\pm$. And observe that, for $r >0$,
	\begin{enumerate}
		\item $\|[\phi_{\lambda_r}]_x\|^2=r$;
		\item $[\phi_{\lambda_r}]_{xx}+\frac{f(\phi_{\lambda_r})}{c(r)}=0$.
	\end{enumerate}
	
	Let $\psi(r)=\phi_{\lambda_r}$. Differentiating with respect to $r$, and representing $\frac{d\psi(r)}{dr}=\dot{\psi}(r)$, we find
	$$
	\dot{\psi}_{xx}(r)+\frac{f'(\psi(r))\dot{\psi}(r)}{c(\|(\psi(r))_x\|^2)}-\frac{f(\psi(r))c'(\|(\psi(r))_x\|^2)}{[c(\|(\psi(r))_x\|^2)]^2}\frac{d}{dr}\|(\psi(r))_x\|^2=0.
	$$

	Now, since
	\begin{equation*}
		\begin{split}
			\frac{d}{dr}\|(\psi(r))_x\|^2&=2\left<(\psi(r))_x,(\dot{\psi}(r))_x\right>=-2\left<(\psi(r))_{xx},\dot{\psi}(r)\right>\\
			&=\frac{2}{c(\|(\psi(r))_x\|^2)}\left<f(\psi(r)),\dot{\psi}(r)\right>
		\end{split}
	\end{equation*}
	we may write
	$$
	\dot{\psi}_{xx}(r)+\frac{f'(\psi(r))\dot{\psi}(r)}{c(\|(\psi(r))_x\|^2)}
	%-\frac{f(\psi(r))c'(\|(\psi(r))_x\|^2)}{[c(\|(\psi(r))_x\|^2)]^2}
	-\frac{2c'(\|(\psi(r))_x\|^2)}{c(\|(\psi(r))_x\|^2)^3}f(\psi(r))\left<f(\psi(r)),\dot{\psi}(r)\right>=0.
	$$
	
	Now, if $L_c:D(L_c)\subset L^2(0,\pi)\to L^2(0,\pi)$ is given by $D(L_c)=H^2(0,\pi)\cap H^1_0(0,\pi)$ and
	$$
	L_cv=v'' +\frac{\lambda f'(\psi(r))}{c(\|(\psi(r))_x\|^2)}v-\frac{2\lambda c'(\|(\psi(r))_x\|^2)}{c(\|(\psi(r))_x\|^2)^3}f(\psi(r))\int_{0}^{\pi} f(\psi(r))v, \ v\in D(L_c),
	$$
	we have that $L_{c}\dot{\psi}(r)=0$ and, since $\dot{\psi}(r) \in H^2(0,\pi)\cap H^1_0(0,\pi)$, it follows that $0 \in \sigma(L_{c})$.

\section{Identifying the equilibria of the nonlocal problem}\label{Identification}

Let us study the sequence of bifurcations for the nonlocal problem \eqref{eq_non-local}. 

\begin{theorem}\label{theo:Identify_equil_nl}

For each positive integer $j$ consider $c_j^+(\cdot)$ and $c_j^-(\cdot)$, the two maps defined above.
For $\nu>0$ and $r>0$, \eqref{eq_non-local} has an equilibrium $\psi$, with $j-1$ zeros in the interval $(0,\pi)$ such that $(\psi)_x(0)>0$ (resp. $(\psi)_x(0)<0$) and  $\|\psi_x\|^2=r$ if and only if $\nu c_j^+(r)=a(r)$  (resp. $\nu c_j^-(r)=a(r)$).
\end{theorem}

\begin{proof} If $\psi$ is an equilibrium of \eqref{eq_non-local}, with $j-1$ zeros in the interval $(0,\pi)$ such that $\psi_x (0)>0$ and  $\|\psi_x\|^2=r$  then $\psi=\phi_{j,\lambda_{j,r}^+}^+$ and $\nu c_j^+(r)=a(r)$. On the other hand,  since $\phi_{j,\lambda_{j,r}^+}^+$ is a solution of \eqref{local_nonlocal} with $\|(\phi_{j,\lambda_{j,r}^+}^+)_x\|^2=r$ and since $\nu c_j^+(r)=a(r)$, $\psi=\phi_{j,\lambda_{j,r}^+}^+$ is an equilibrium of \eqref{eq_non-local}.  A similar argument is used for $c_j^-$. $\ $
\end{proof}

\begin{remark}
For each positive integer $k$,  if $\nu > k^2a(0)$ there are at least $2k+1$ equilibria of the non-local problem \eqref{eq_non-local}. That is an immediate consequence of the fact that the functions $c_j^\pm: [0,+\infty) \to (0, \frac{1}{j^2}]$ are continuous, $\nu c_j^\pm(0) =  \frac{\nu}{j^2}>a(0)$, $c_j^\pm(r)\stackrel{r\to +\infty}{\longrightarrow} 0$, $1\leq j\leq k$, and $a:[0,+\infty)\to [m,M]$ is continuous. In particular, if $a$ is non-decreasing we have exactly $2k+1$ equilibria of \eqref{eq_non-local}.
\end{remark}

\section{The hyperbolicity and Morse Index of equilibria}\label{hyperbolicity}

Consider the auxiliary initial boundary value problem \eqref{eq_nl_changed}
related to \eqref{eq_non-local} by a solution dependent change of the time scale. As we have mentioned before, both problems have exactly the same equilibria and, as in \cite{Estefani2}, the spectral analysis of the self-adjoint operator associated to the linearization of \eqref{eq_nl_changed} around an equilibrium $\psi$ will determine its stability and hyperbolicity properties.

%%%%\bigskip

The linearization of \eqref{eq_nl_changed} around an equilibrium $\psi$ is given by the equation
\begin{equation*}
	v_t = Lv,
\end{equation*}
where $D(L)=H^2(0,\pi)\cap H^1_0(0,\pi)$ and 
$$
Lv=v'' +\frac{\nu f'(\psi)}{a(\|\psi_x\|^2)}v-\frac{2\nu^2 a'(\|\psi_x\|^2)}{a(\|\psi_x\|^2)^3}f(\psi)\int_{0}^{\pi} f(\psi)v, \ v\in D(L).
$$

Given an equilibrium $\psi\neq 0$ of \eqref{eq_nl_changed}, let $r=\|\psi_x\|^2$ and let $k$ be the positive integer such that $\psi$ vanishes $k-1$ times in the interval $(0,\pi)$.  Then if $\psi_x(0)>0$, and if we consider $\lambda_{k,r}^+=(c_k^+(r))^{-1}$ then $\psi=\phi_{k,{\lambda_{k,r}^+}}^+$. Similarly, if $\psi_x(0)<0$, and if we consider $\lambda_{k,r}^-=(c_k^-(r))^{-1}$ then $\psi=\phi_{k,{\lambda_{k,r}^-}}^-$.
For simplicity of notation we will write $c(\cdot)$ instead of $c_k^\pm(\cdot)$, $\lambda_r$ instead of $\lambda_{k,r}^\pm$ and $\phi_{\lambda_r}$ instead of $\phi_{k,\lambda_{k,r}^\pm}^\pm$ for the remainder of this section.

\subsection{Hyperbolicity} 
This section is concerned with the characterization of hyperbolicity for the equilibria of \eqref{eq_nl_changed} given by the theorem below. 

\begin{theorem}\label{theo:hyp_semilinear}
With the notation above, the equilibrium $\psi$ of  \eqref{eq_nl_changed} is not hyperbolic if, and only if, 
	$a'(\|\psi_x\|^2)= \nu c'(\|\psi_x\|^2)$.
\end{theorem}
\begin{proof} $(\Leftarrow)$ Suppose initially that $a'(\|\psi_x\|^2)= \nu c'(\|\psi_x\|^2)$. Let $r= \|\psi_x\|^2$. In the notation above, we have that {$\psi = \phi_{\lambda_r}$}.

 Recall that, as we have seen at the end of Section \ref{CI-properties}, $v=\frac{d}{dr}\phi_{\lambda_r}$ satisfies	
\begin{equation*}
v_{xx}+\frac{f'(\phi_{\lambda_r})}{c(r)}v-\frac{2c'(r)}{c(r)^3}f(\phi_{\lambda_r})\int_0^\pi f(\phi_{\lambda_r})v=0.
\end{equation*}

From Theorem \ref{theo:Identify_equil_nl} we have that $a(r)=\nu c(r)$. Since, $\|\psi_x\|^2=r$, $\psi=\phi_{\lambda_r}$ and $a'(r)=\nu c'(r)$ we have
\begin{equation*}
	v_{xx}+\frac{\nu f'(\psi)}{a(\|\psi_x\|^2)}v - \frac{2\nu^2a'(\|\psi_x\|^2)}{a(\|\psi_x\|^2)^3} f(\psi)\int_0^\pi f(\psi)v=0.
\end{equation*}

Therefore, $0$ is an eigenvalue of $L$, which implies that $\phi$ is not a hyperbolic equilibrium.

%%%%\bigskip 
	
$(\Rightarrow)$ Assume that we find a $0\neq u \in H^2(0,\pi)\cap H^1_0(0,\pi)$ satisfying
$$
u_{xx}+\frac{\nu f'(\psi)}{a(\|\psi_x\|^2)}u-\frac{2\nu^2\alpha a'(\|\psi_x\|^2)}{a(\|\psi_x\|^2)^3}f(\psi)=0	
$$
for $\alpha ={\displaystyle\int_0^\pi} f(\psi(s))u(s)ds$.

Now, since $a(r)=\nu c(r)$ and $\phi_{\lambda_r}\!=\psi$, $v=\frac{d}{dr}\phi_{\lambda_r}$ satisfies
$$
v_{xx}+\frac{\nu f'(\psi)}{a(\|\psi_x\|^2)} v-\frac{2\nu^3\beta c'(\|\psi_x\|^2)}{a(\|\psi_x\|^2)^3}f(\psi)=0,
$$	
for $\beta = {\displaystyle\int_0^\pi} f(\psi(s))v(s)ds$. 

Consequently, $w=\beta \nu c'(r) u-\alpha a'(r) v$ is the solution of
\begin{equation}\label{eq:prob_lin_local}
	\left\{
	\begin{aligned}
		& w_{xx}+\frac{\nu f'(\phi_{\lambda_r})}{a(\|\phi_{\lambda_r}\|^2)}w=0	\hbox{ or } w_{xx}+\lambda_r f'(\phi_{\lambda_r})w=0,	\\
		& w(0)=w(\pi)=0,
	\end{aligned}\right.
\end{equation}
which means $w\equiv 0$. Thus $\beta\nu c'(r) u-\alpha a'(r) \dot{\phi}=0$ and, by multiplying both sides of equality by $f(\phi)$ and integrating from $0$ to $\pi$, we find
$$\alpha\beta\nu c'(r)=\alpha \beta a'(r).$$

Clearly, $\alpha \beta \neq 0$. Otherwise, either $u$ or $v$ should be a solution of \eqref{eq:prob_lin_local}, that is, either $u=0$ or $v=0$, which would be a contradiction.

Therefore, we conclude that $a'(r)=\nu c'(r)$.

\end{proof}

\subsection{Morse Index} Now we analyze what happens to the dimension of the unstable manifolds for the equilibria of \eqref{eq_nl_changed} as they bifurcate.

For $\varepsilon \in \mathbb{R}$, define the operator $L_\varepsilon: H^2(0,\pi)\cap H^1_0(0,\pi)\subset L^2(0,\pi)\to L^2(0,\pi) $ by
\begin{equation*}
	L_\varepsilon u(x) = u'' + p(x)u +\varepsilon q(x)\int_{0}^{\pi} q(s)u(s)ds,
\end{equation*}
where $p,q:[0,\pi]\to \mathbb{R}$ are continuous functions with $q\not\equiv 0$.

When $\varepsilon=0$, $L_0u=u''+p(x)u$ is a 
Sturm-Liouville operator. Hence, $L_0$ is a self-adjoint with compact resolvent and its spectrum consists of a decreasing sequence of simple eigenvalues, that is,
$$
\sigma(L_0)=\{\gamma_j: j=1,2,3\cdots \} 
$$
with, $\gamma_j > \gamma_{j+1}$ and $\gamma_j \longrightarrow -\infty$ as  $j \rightarrow +\infty$.

%%%%\bigskip 

Note that for all $\varepsilon \in \mathbb{R}$, we can decompose $L_\varepsilon$ as sum of two operators:
$$
L_\varepsilon u=L_0 u + \varepsilon Bu,
$$
where $Bu=q(x){\displaystyle\int_0^\pi} q(s)u(s)ds$, for all $u \in H^2(0,\pi)\cap H^1_0(0,\pi)$, is a bounded operator with rank one. It is easy to see that $L_\varepsilon$ is also self-adjoint with compact resolvent. 

%%%%\bigskip

Then, we write $\{ \mu_j(\varepsilon): j=1,2,3,\cdots\}$ to represent the eigenvalues of $L_\varepsilon$, ordered in such a way that, for $j=1,2,3,\cdots$, the function $\mathbb{R}\ni \varepsilon \mapsto \mu_j(\varepsilon)\in \mathbb{R}$ satisfies $\mu_j(0)=\gamma_j$.

Throughout this paper we will use, in an essential way, Theorems 3.4 and 4.5 of \cite{DavidsonDodds}. We summarize these results next.
\begin{theorem}\label{theo:DD}
	Let $L_\varepsilon$ and $\{ \mu_j(\varepsilon): j=1,2,3,\cdots \}$ be as above. The following holds:
	\begin{itemize}
		\item[i)] For all $j=1,2,3,\cdots $, the function  $\mathbb{R}\ni \varepsilon \mapsto \mu_j(\varepsilon)\in \mathbb{R}$ is non-decreasing.
		\item[ii)] If for some $j=1,2,3\cdots $ and $\varepsilon \in \mathbb{R}$, $\mu_j(\varepsilon)\notin \{\gamma_k:k =1,2,3,\cdots \}$, then $\mu_j(\varepsilon)$ is a simple eigenvalue of $L_\varepsilon$. 
	\end{itemize}
\end{theorem}

%%%%\bigskip

We wish to determine the Morse Index of the equilibria by looking carefully to the points where the graphs of the functions $a(\cdot)$ and $\nu c(\cdot)$ intercept, that is, depending on how they curves intersect we will be able to determine the Morse Index of the equilibria. Recall that the function $c(\cdot)$ is in fact $c_j^+(\cdot)$ or $c_j^-(\cdot)$ which is associated to an equilibrium $\phi_{j,\lambda_r^+}^+$ or $\phi_{j,\lambda_r^-}^-$ which change sign $j-1$ times in the interval $(0,\pi)$. The intersection of the graphs of $a(\cdot)$ and $\nu c_j^\pm(\cdot)$ necessarily gives rise to an equilibrium that changes sign $j-1$ times for \eqref{eq_non-local}. Hence, as $\nu$ increases, if the first intersection between the graphs of $a(\cdot)$ and $\nu c_j^\pm(\cdot)$ happens with a value of $r\neq 0$ and before the intersection with $r=0$, we must have at least one saddle-node bifurcation that precedes the pitchfork bifurcation from zero  (see, for instance, Example \ref{example:a_1}%Figure \ref{figure-ibgac}
).

This is the main result of this section:

\begin{theorem}\label{theo:Morse_index}
Suppose that $\psi$ is an equilibrium of \eqref{eq_nl_changed} with $k-1$ zeros in $(0,\pi)$ for some positive integer $k$. Let $r=\|\psi_x\|^2$ and $\lambda_r$ such that $\psi=\phi_{k,\lambda_r}^+$ (resp.  $\psi=\phi_{k,\lambda_r}^-$). If we denote $c_k^\pm(\cdot)$ by $c$, then
	\begin{enumerate}[label=$(\roman*)$]
	\item If $a'(\|\psi_x\|^2) > \nu c'(\|\psi_x\|^2)$, then $\psi$ is hyperbolic and its Morse index is $k-1$.
	\item If $a'(\|\psi_x\|^2) < \nu c'(\|\psi_x\|^2)$, then $\psi$ is hyperbolic and its Morse index is $k$.
	\end{enumerate}
\end{theorem}
\begin{proof}

The hyperbolicity follows from Theorem \ref{theo:hyp_semilinear}. Define the operator
$$
L_\varepsilon v=v'' +\frac{\nu f'(\psi)}{a(\|\psi'\|^2)}v+\varepsilon f(\psi)\int_{0}^{\pi} f(\psi)v,
$$
$v\in D(L_\varepsilon)=H^2(0,\pi)\cap H^1_0(0,\pi)$, for each $\varepsilon >0$.

%%%\bigskip

Note that $L_0$ is the linearization of \eqref{eq:C-I} at $\psi$ for the parameter $\nu_0=\frac{\nu}{a(\|\psi_x\|^2)}$.
The spectrum of $L_0$ is given by an unbounded ordered sequence $\{\mu_j(0)\}_{j \in \mathbb{N}}$ of simple eigenvalues, that is, 
$$
\mu_1(0) >\mu_2(0) > \dots > \mu_{k-1}(0)> \mu_k(0) > \mu_{k+1}(0)> \dots
$$
Since $0\notin \sigma(L_0)$ we may have that $0> \mu_j(0)$, for all $j=1,2,3\cdots$, or there is a positive integer 
$k$ such that $\mu_{k-1}(0)> 0> \mu_k(0)$.

For $\tilde{\varepsilon}=-\frac{2 c'(\|\psi_x\|^2)}{c(\|\psi_x\|^2)^3}=-\frac{2 \nu^3 c'(\|\psi_x\|^2)}{a(\|\psi_x\|^2)^3}$, we have $0\in \sigma(L_{\tilde{\varepsilon}})$ and $0$ is a simple eigenvalue, by Theorem \ref{theo:DD}. Using the same reasoning applied in the proof of the second part of Theorem \ref{theo:hyp_semilinear} we can show that if $0 \in \sigma(L_\varepsilon)$, then $\varepsilon =\tilde{\varepsilon}$.

Using Theorem \ref{theo:DD}, part $i)$, we deduce that $\mu_j(\varepsilon)>0$, $j=1,\cdots, k-1$, for all $\varepsilon\geq 0$. Since $0\in \sigma(L_\varepsilon)$ if and only if $\varepsilon=\tilde{\varepsilon}>0$ we must have that $\mu_j(\varepsilon)>0$,  $j=1,\cdots,k-1$, for all $\varepsilon<0$. That means at least $k-1$ eigenvalues are positive, for all $\varepsilon \in \mathbb{R}$.

Also, since $\mu_k(0)>\mu_j(0)$, for all $j>k$, 
$\mu_k(\cdot)$ is increasing and $L_0$ does not have an eigenvalue in the interval $(\mu_k(0),0]$ we have that $\mu_k(\tilde{\varepsilon})=0$. Otherwise $\mu_j(\varepsilon)=\mu_k(\varepsilon)\in (\mu_k(0),0]$ for some $j>k$ and $\varepsilon\in (0,\tilde{\varepsilon}]$ which is not possible by Theorem \ref{theo:DD}, part $ii)$. Since $0\notin \sigma(L_0)$, $\mu_j(\varepsilon)<0$ for all $\varepsilon\in \mathbb{R}$ and $j>k$.

As a consequence, the number of positive eigenvalues of $L_\varepsilon$ is $k-1$ if $\varepsilon<\tilde{\varepsilon}$ and $k$ if $\varepsilon>\tilde{\varepsilon}$ (see Figure 1).

\begin{figure}
\includegraphics[scale=0.6]{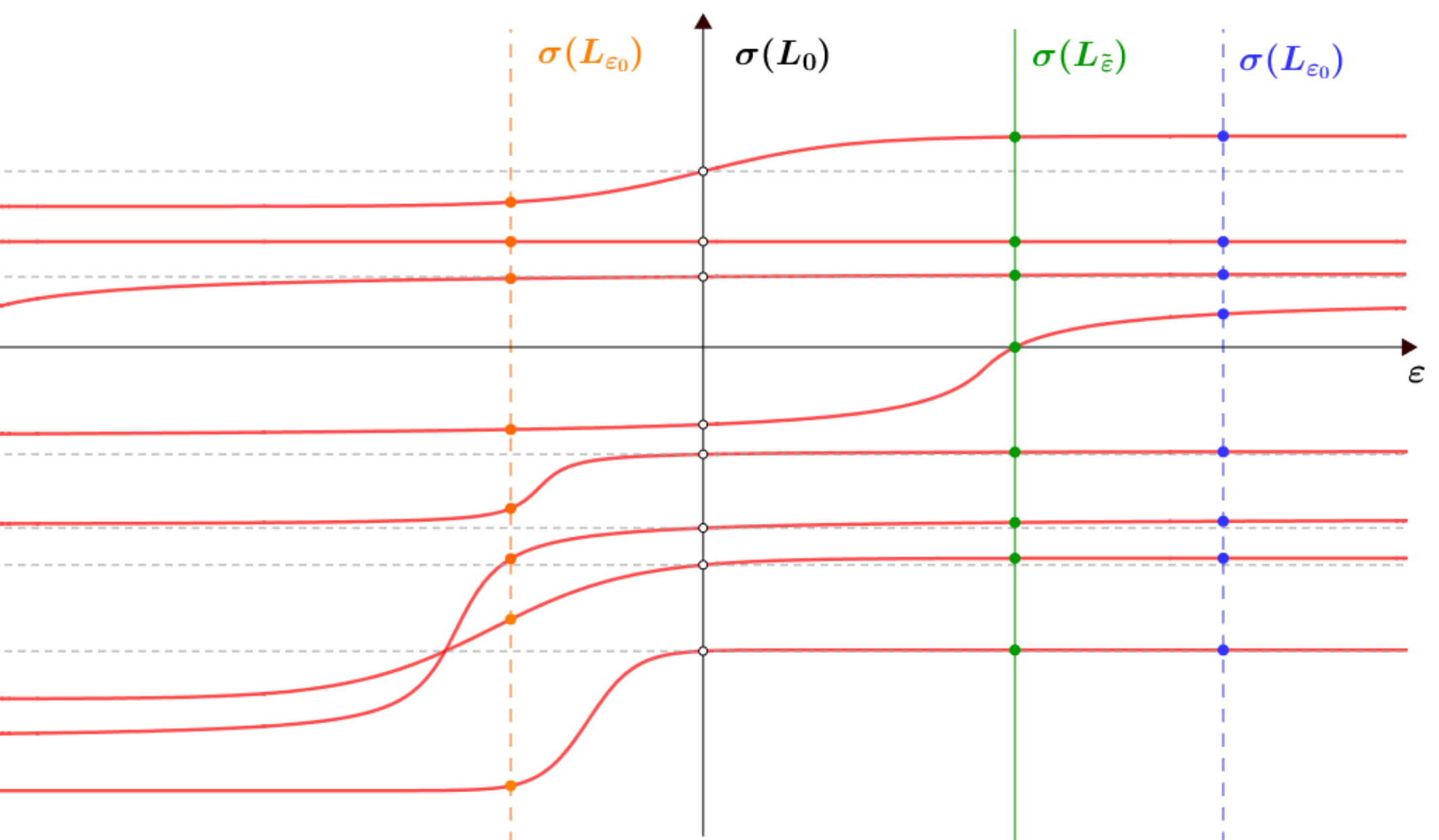}
\caption{Spectrum of $L_\epsilon$}
\end{figure}

Let $\varepsilon_0=-\frac{2\nu^2 a'(\|\psi_x\|^2)}{a(\|\psi_x\|^2)^3}$. Then, if $a'(\|\psi_x\|^2) > \nu c'(\|\psi_x\|^2)$, we have that $\varepsilon_0<\tilde{\varepsilon}$ and $L_{\varepsilon_0}$ has  exactly $k-1$ positive eigenvalues and $0\notin\sigma(L_{\varepsilon_0}$). On the other hand, if
$a'(\|\psi_x\|^2) < \nu c'(\|\psi_x\|^2)$, we have that $\varepsilon_0>\tilde{\varepsilon}$ and $L_{\varepsilon_0}$ has  exactly $k$ positive eigenvalues and $0\notin\sigma(L_{\varepsilon_0})$.  
\end{proof}

\section{Analyzing the attractor for a few examples}\label{examples}

Denote by $c^L_{j,\pm}(\cdot)$, $L>0$, $j \in \mathbb{N}$, the function $c_j^\pm(\cdot)$ related to the equilibria that have $j-1$ zeros in $(0,\pi)$ of the problem
\begin{equation}\label{eq:CI_intL}
\left\{
\begin{aligned}
&u_{xx}+ \lambda f(u)=0, x \in (0,L),\\
&u(0)=u(L)=0,
\end{aligned}
\right.
\end{equation}
for $\lambda >0$ a parameter.

Recall that the following holds.
\begin{lemma}\label{lemma:symm_even_equil}
Consider $f \in C^2(\mathbb{R})$ satisfying \eqref{eq_prop_f}. If $j \in \mathbb{N}$, $j\geq 2$,  and $\phi_{j}$ is an equilibrium of \eqref{eq:C-I} with $j-1$ zeros in $(0,\pi)$, then $\phi_{2j}$ is $\frac{\pi}{j}$ periodic. In addition, if $f$ is odd then $\phi_j(\frac{\pi}{j}+x)=-\phi_j(\frac{\pi}{j}-x)$, for $x \in [0,\frac{\pi}{j}]$.
\end{lemma}

\begin{proposition}\label{prop:rel_c_j} If $f \in C^2(\mathbb{R})$ and satisfies \eqref{eq_prop_f}, then:
 \begin{enumerate}[label=$(\roman*)$]
  \item $c_{j,\pm}^L(r)=\left(\frac{L}{\pi}\right)^2c_{j,\pm}^\pi(\frac{Lr}{\pi})$, for all $r\in \mathbb{R}^+$, $j=1,2,3\cdots $. 
\item For all $r\in \mathbb{R}^+$, $c_{2j,\pm}^\pi(r) = \frac{1}{j^2} c_{2,\pm}^\pi(\frac{r}{j^2})$.

\hspace{-45pt}If we also assume that $f$ is odd, then:
  \item $c^L_{j,+}(\cdot)=c^L_{j,-}(\cdot)$ and $c_{j,\pm}^\pi(r)=c_{1,\pm}^\frac{\pi}{j}(\frac{r}{j})$, for all $r\in \mathbb{R}^+$ and $j \in \mathbb{N}$.
  \item $c^L_{j,+}(\cdot)=c^L_{j,-}(\cdot)$ and $c_{j,\pm}^\pi(r)=\frac{1}{j^2}c_{1,\pm}^\pi(\frac{r}{j^2})$, for all $r \in \mathbb{R}^+$ and $j \in \mathbb{N}$.
 \end{enumerate}
\end{proposition}
\begin{proof} The proof follows by a simple change of variables. 

\begin{enumerate}[label=$(\roman*)$] \item Let $r\in \mathbb{R}^+$. In what follows we fix one of the symbols $+$ or $-$ and omit it in the notation. If $c_j^L(r) = \frac{1}{\lambda_r}$, then there is a $\phi \in C^2(0,L)$, with $\|\phi_x\|^2=r$, such that $\phi\neq 0$ in $(0,\pi)$ and satisfies \eqref{eq:CI_intL} with $\lambda$ replaced by $\lambda_r$.

For $x \in [0,\pi]$, define $\psi(x)=\phi(\frac{Lx}{\pi})$. Then $\psi$ satisfies 
$$
\psi_{xx}(s)=\left(\frac{L}{\pi}\right)^2\phi_{xx}\left(\frac{Ls}{\pi}\right)=-\left(\frac{L}{\pi}\right)^2\lambda_rf\left(\phi\left(\frac{Ls}{\pi}\right)\right).
$$

In other words, $\psi$ is a solution of \eqref{eq:CI_intL} with $L$ replaced by $\pi$ and $\lambda$ replaced by $\left(\frac{L}{\pi}\right)^2\lambda_r$.
Also, 
$$
\|\psi_x\|^2 = \int_0^\pi (\psi_x(s))^2ds=\int_0^\pi \left(\frac{L}{\pi}\right)^2 \left(\phi_x\left(\frac{Ls}{\pi}\right)\right)^2ds = \frac{L}{\pi}\int_0^L \phi_x(u)^2du=\frac{Lr}{\pi}.
$$

Hence, by definition of $c_j^\pi$, we conclude that $c_j^\pi(\frac{Lr}{\pi})=(\frac{\pi}{L})^2\frac{1}{\lambda_r}$.

Therefore, $c_j^L(r)=\frac{1}{\lambda_r}=\left(\frac{L}{\pi}\right)^2 c_j^\pi(\frac{Lr}{\pi})$. Since $r\in \mathbb{R}^+$ is arbitrary, the result follows.

\item	Once again, we fix one of the symbols $+$ or $-$ and omit it in the notation. Let $r> 0$ and $j \in \mathbb{N}$, $j\geq 2$. 
By the definition, $c_{2j}^\pi(r)=\frac{1}{\lambda_r}$ implies that there is a $\phi$, with $2j-1$ zeros in $(0,\pi)$, an equilibrium of \eqref{eq:C-I} when $\lambda =\lambda_r$ and satisfying $\|\phi_x\|^2=r$.

By Lemma \ref{lemma:symm_even_equil}, we have that 
$r=\int_0^\pi (\phi_x(s))^2 ds = j \int_0^\frac{\pi}{j} (\phi_x(s))^2 ds$.

Hence $\psi = \phi\big|_{[0,\frac{\pi}{j}]}$ is the solution of \eqref{eq:CI_intL} that changes sing one time for $L=\frac{\pi}{j}$ and $\|\psi_x\|_{L^2(0,\frac{\pi}{2})}=\frac{r}{j}$. 

Therefore, $c_2^\frac{\pi}{j}(\frac{r}{j})=\frac{1}{\lambda_r}=c_{2j}^{\pi}(r)$. By the previous item, the desired result follows.

\item Fix $j \in \mathbb{N}$ and $r \in \mathbb{R}^+$. If $c_j^\pi(r)=\frac{1}{\lambda_r}$, then there is $\phi \in C^2(0,\pi)$ with $j-1$ zeros in $(0,\pi)$, with $\|\phi_x\|^2=r$, and satisfying \eqref{eq:C-I}.  Since $f$ is odd, $\phi$ has a lot of symmetries and 
$$ r = \int_0^\pi (\phi_x(s))^2 ds =j \int_0^\frac{\pi}{j} (\phi_x(s))^2 ds.$$

Consider $\psi=\phi\big|_{[0,\frac{\pi}{j}]}$. Then, we have $\psi>0$ in $(0,\frac{\pi}{j})$,  $\|\psi_x\|^2_{L^2(0,\frac{\pi}{j})}=\frac{r}{j}$, and $\psi$ satisfies \eqref{eq:CI_intL}, for $L=\frac{\pi}{j}$ and $\lambda=\lambda_r$. Hence, by the definition of $c_1^\frac{\pi}{j}$, we find $c_1^\frac{\pi}{j}(\frac{r}{j})=\frac{1}{\lambda_r}=c_j^\pi(r)$.

\item It follows from the previous items.
\end{enumerate}
\end{proof}

The result from Proposition \ref{prop:rel_c_j} provides a very good understanding of the bifurcations of equilibria for \eqref{eq_non-local} with particular emphasis to the case of suitably large $j\in \mathbb{N}$. We remark that, if $f$ is odd, for large values of $j$, the functions $j^2c_j^{\pm}$ are very slowly decreasing. 

%%%%\bigskip

Next we exhibit a few pictorial examples of possible bifurcations that will happen depending on our choice of the functions $a$ and $f$.

\begin{example}\label{example:a_1}
	Consider in this example the function $a=a_1$ as in Figure \ref{fig:2}:

%%%%% SMALL
\begin{figure}[h!]
	\centering
	\includegraphics[scale=0.27]{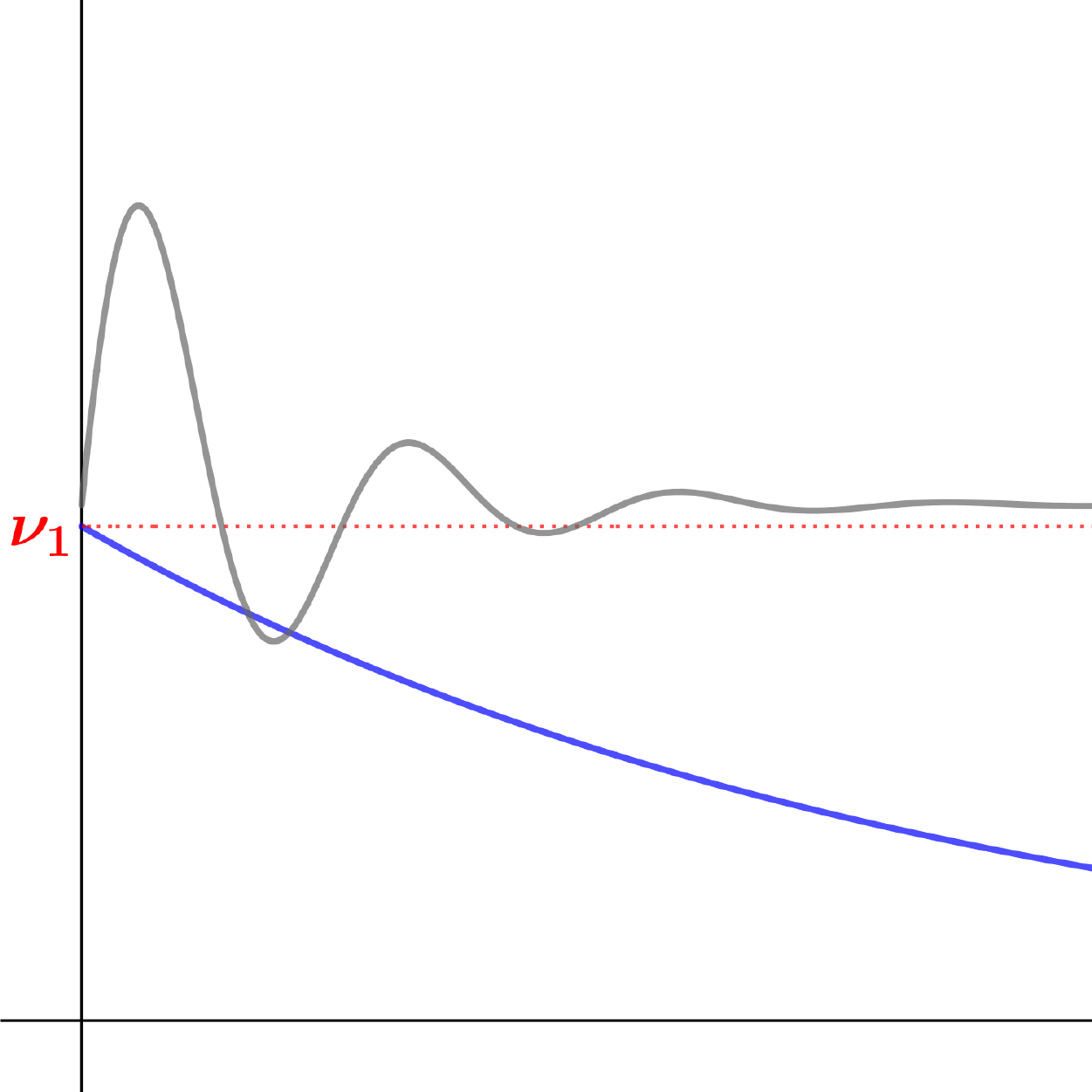} \hfill 
	\includegraphics[scale=0.27]{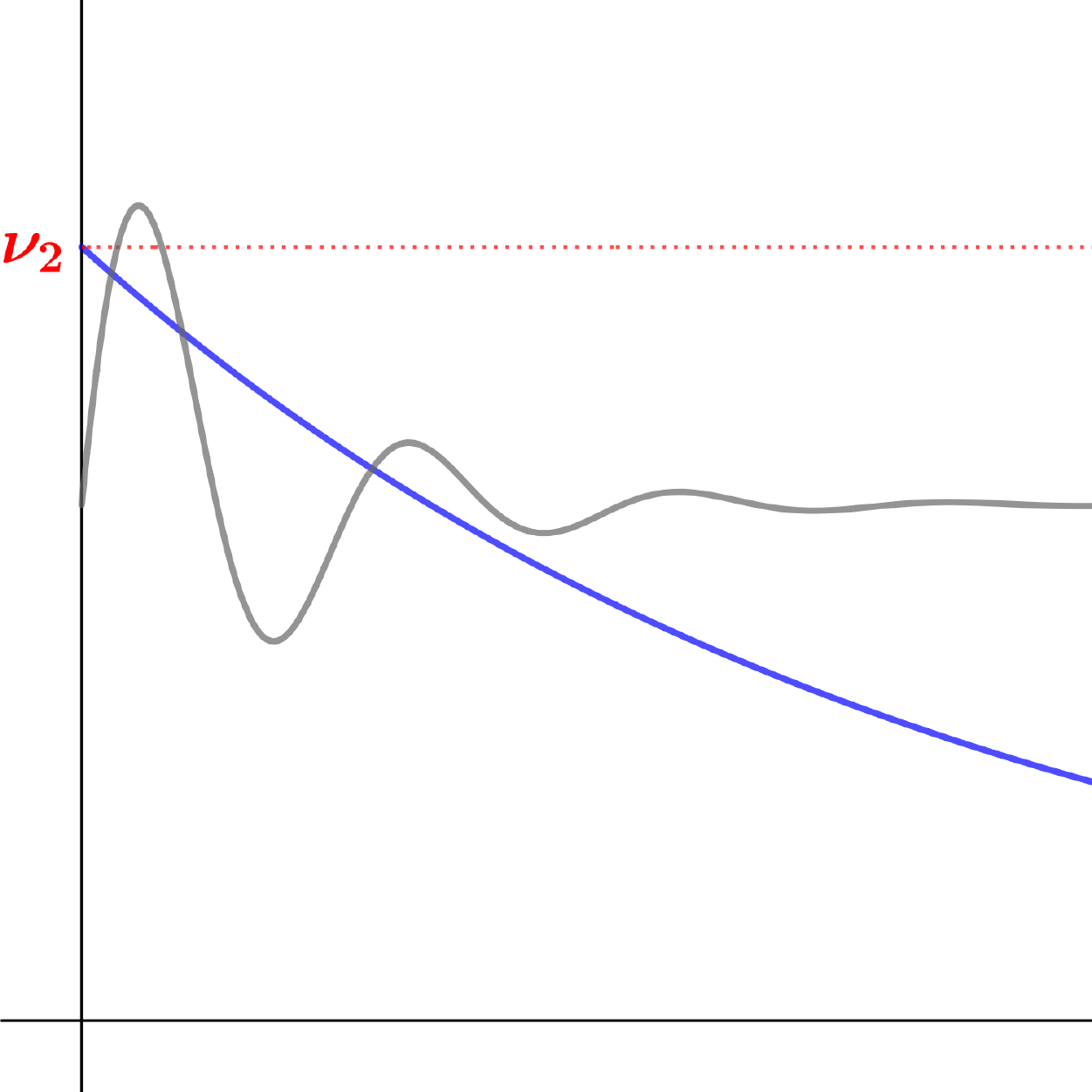} \hfill  
	\includegraphics[scale=0.27]{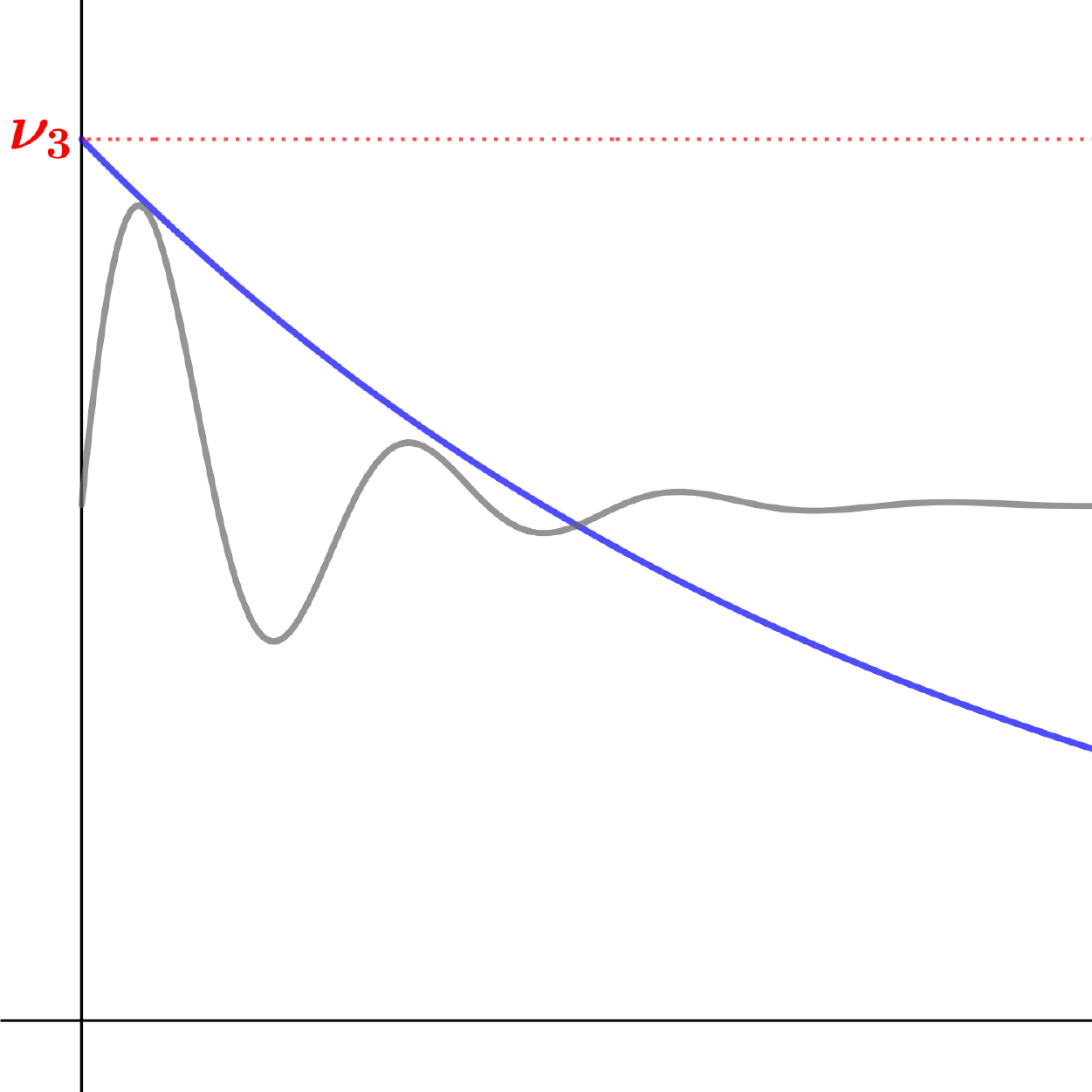} \hfill
	\includegraphics[scale=0.27]{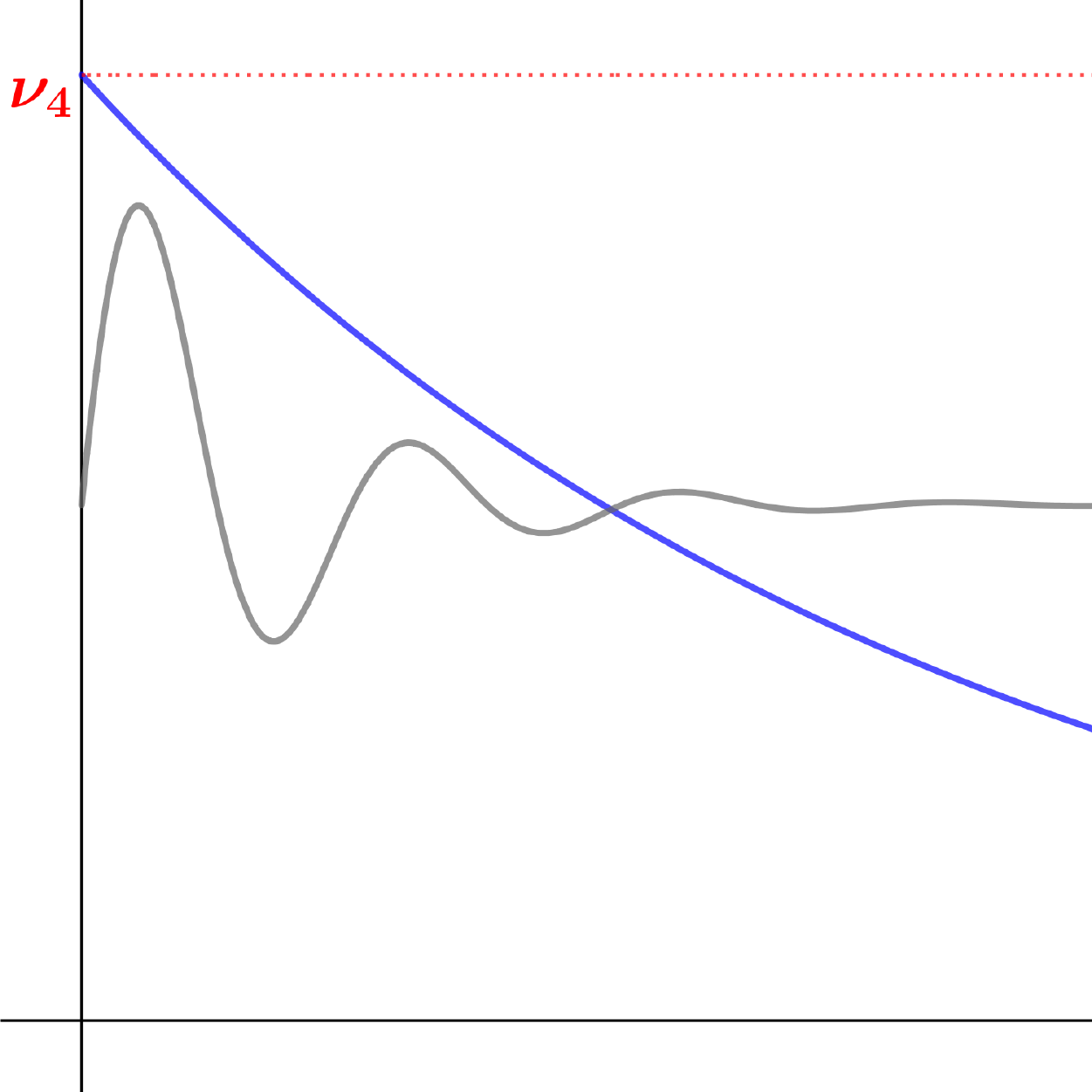} 
	\caption{Graphs of $a_1$ (in gray) and $\nu c_1^\pm$ (in blue) for different choices of $\nu$}\label{fig:2}
\end{figure}

%\begin{figure}[h!]
%	\includegraphics[scale=0.25]{A1.pdf} \hspace{0.5cm}
%	\includegraphics[scale=0.25]{A2.pdf} \hspace{0.5cm}	
%	\includegraphics[scale=0.25]{A3.pdf} 
%	
%	%%%
%	\vspace{0.5cm}
%	%%%%
%	\includegraphics[scale=0.25]{A4.pdf}
%	\caption{Graphs of $a_1$ and $\nu c_j^\pm$ for different choices of $\nu$}\label{fig:2}
%\end{figure}

\begin{multicols}{2}
\begin{minipage}[l]{9cm}
	
	In that case, the bifurcation from zero is a supercritical pitchfork bifurcation and four other saddle-node bifurcations occur, two subcritical and two supercritical. The bifurcation curve looks like this:
\end{minipage}	

\begin{minipage}{6cm}
\begin{flushright}
%\begin{figure}[htpb]
%	\centering
	\vspace{-0.2cm}
	\includegraphics[scale=0.2]{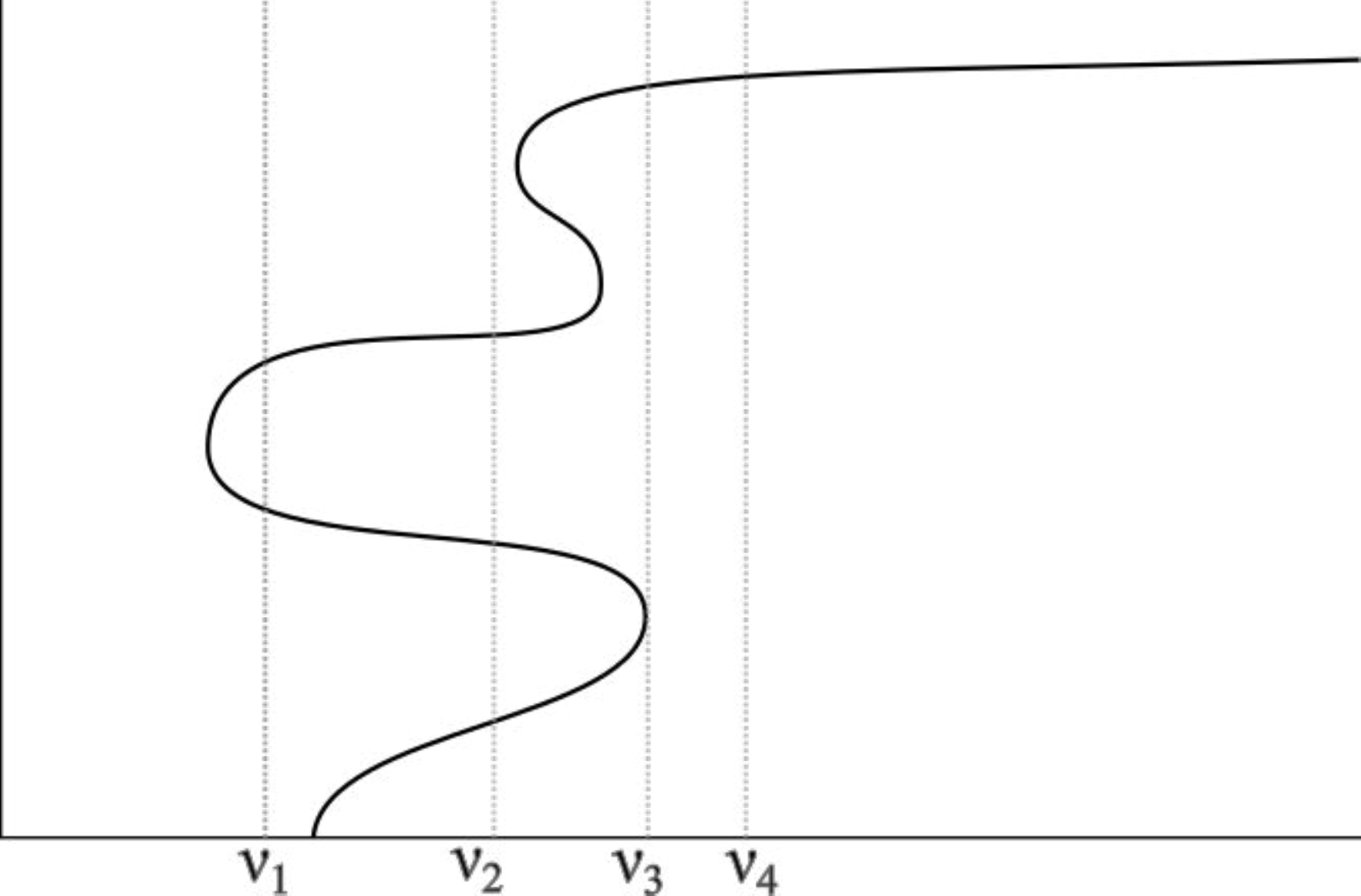}
%	\caption{Bifurcation curve. Intersections between the graphs of $a_1$ and $\nu c_j^+$}	\label{figure-ibgac}
%\end{figure}

\end{flushright}\end{minipage}

\end{multicols}
\end{example}

\begin{example} Consider in this example the function $a=a_2$, with graph pictured in gray, in Figure \ref{fig:graph_a2}:

	\begin{figure}[h!]
		\centering
	\includegraphics[scale=0.27]{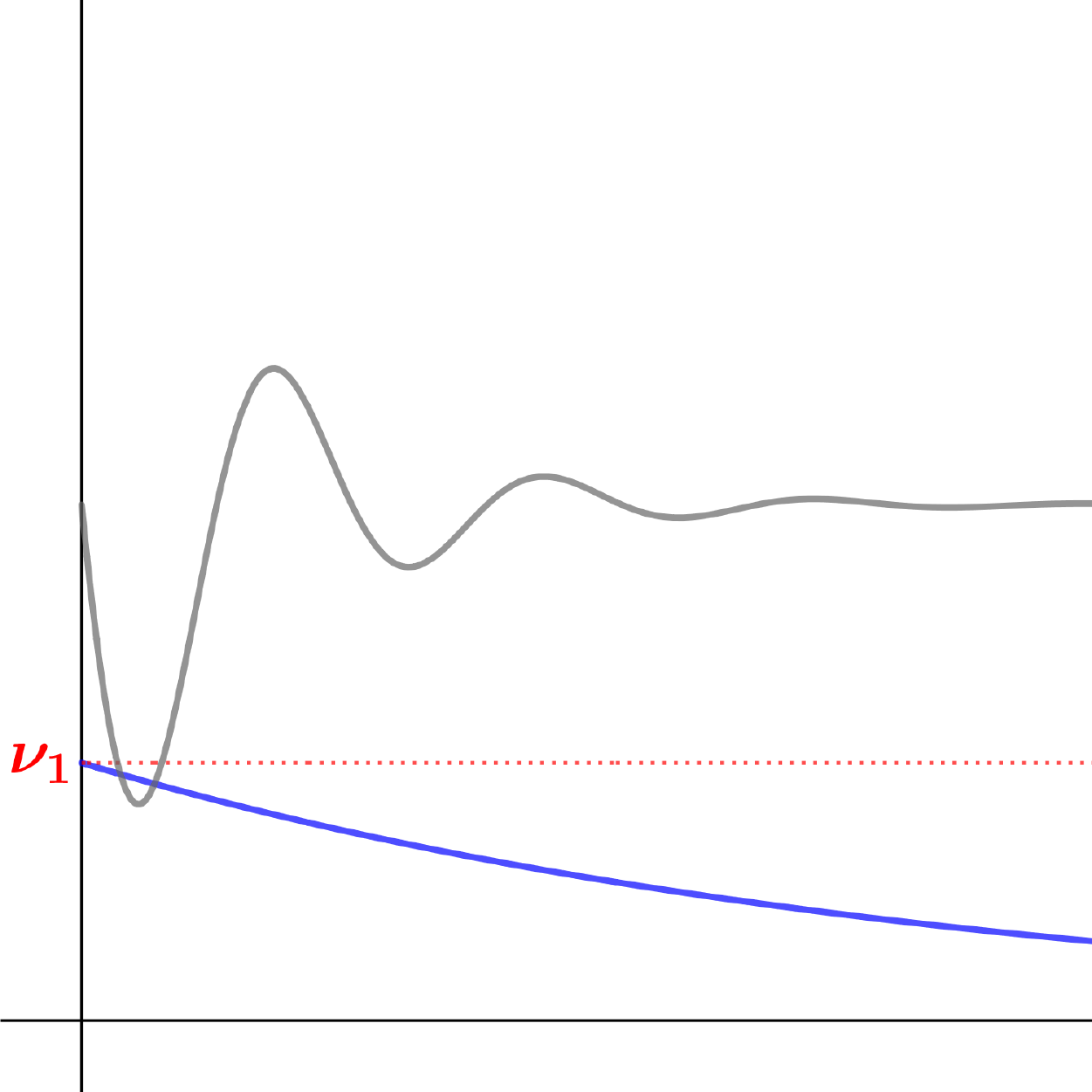} \hfill
	\includegraphics[scale=0.27]{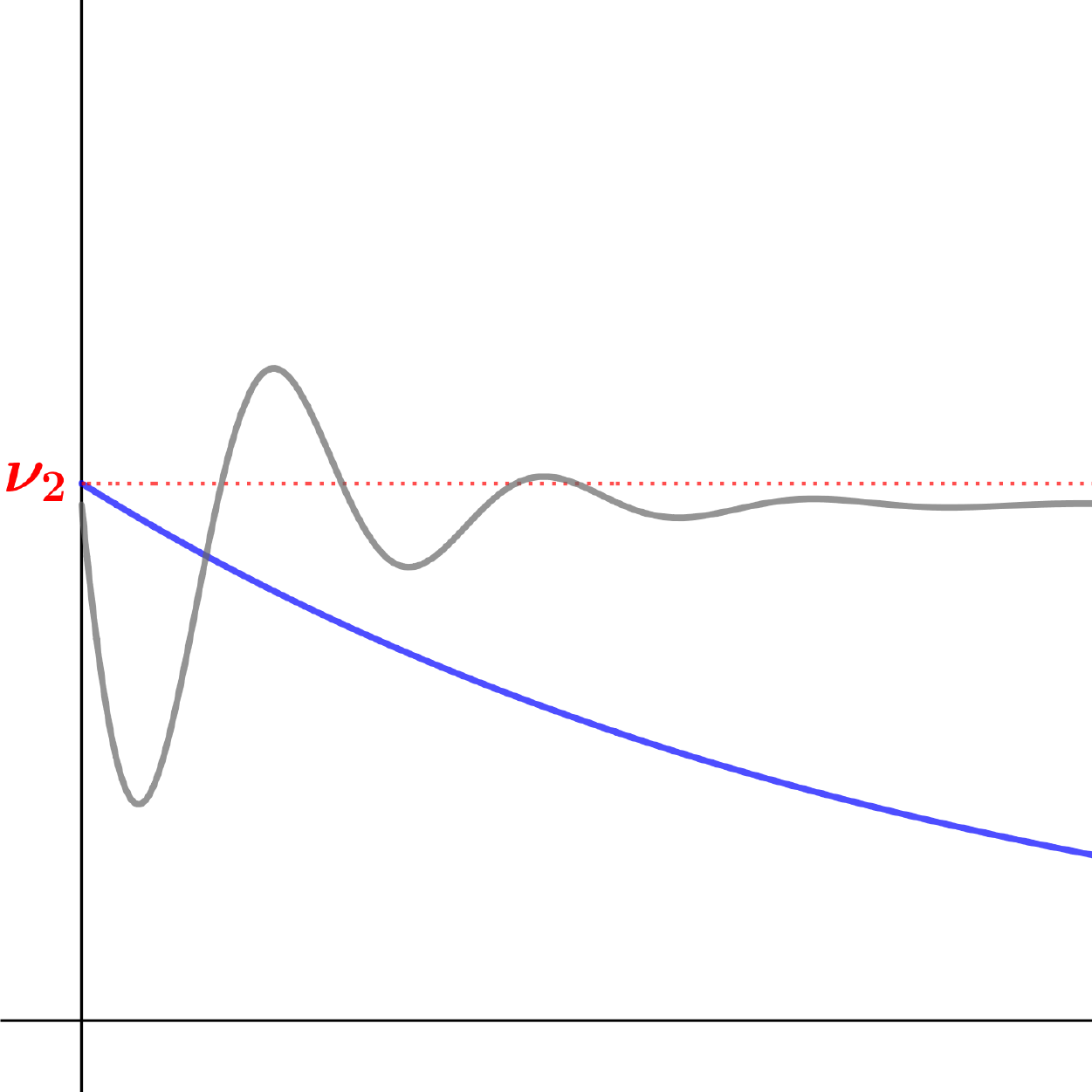} \hfill
	\includegraphics[scale=0.27]{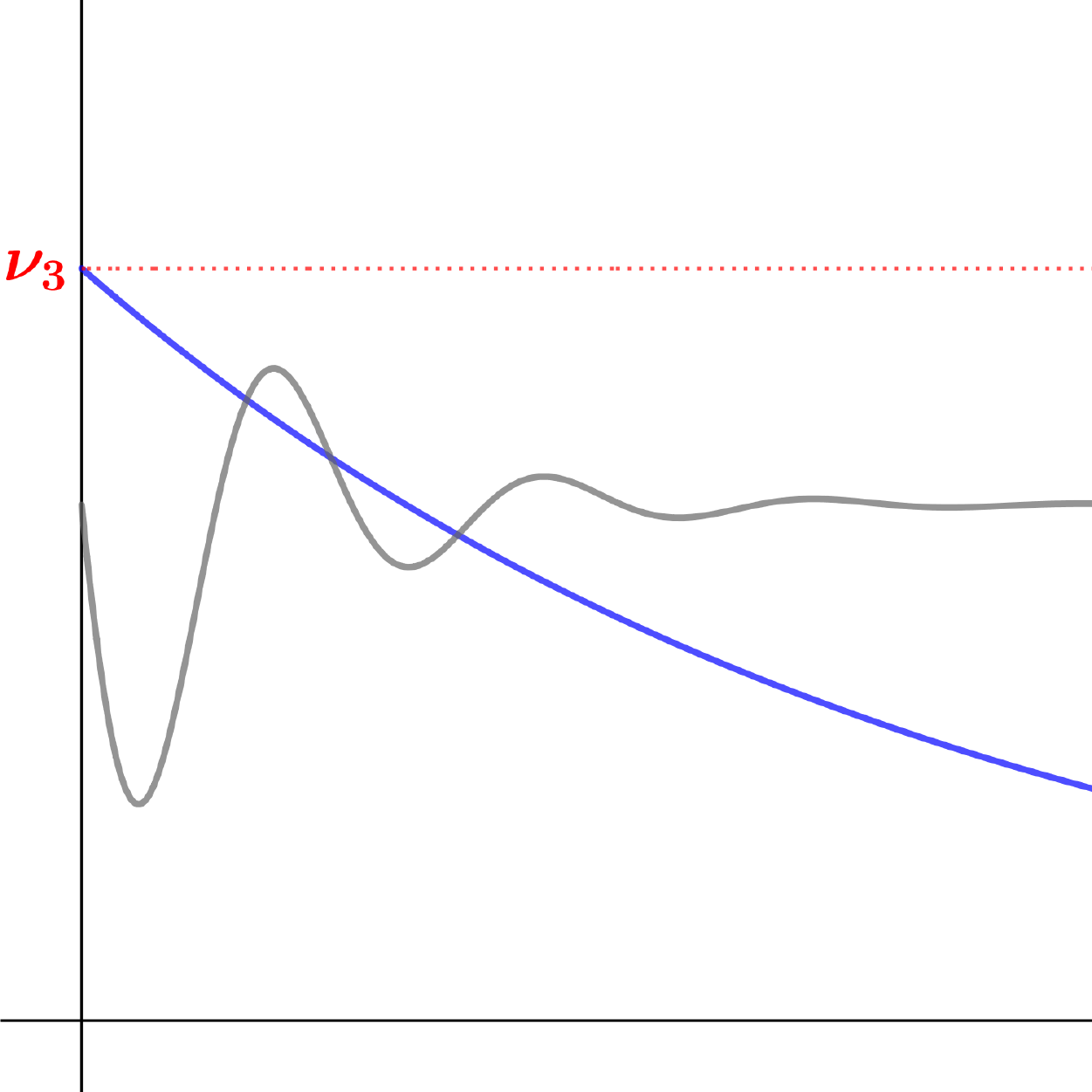} \hfill
	\includegraphics[scale=0.27]{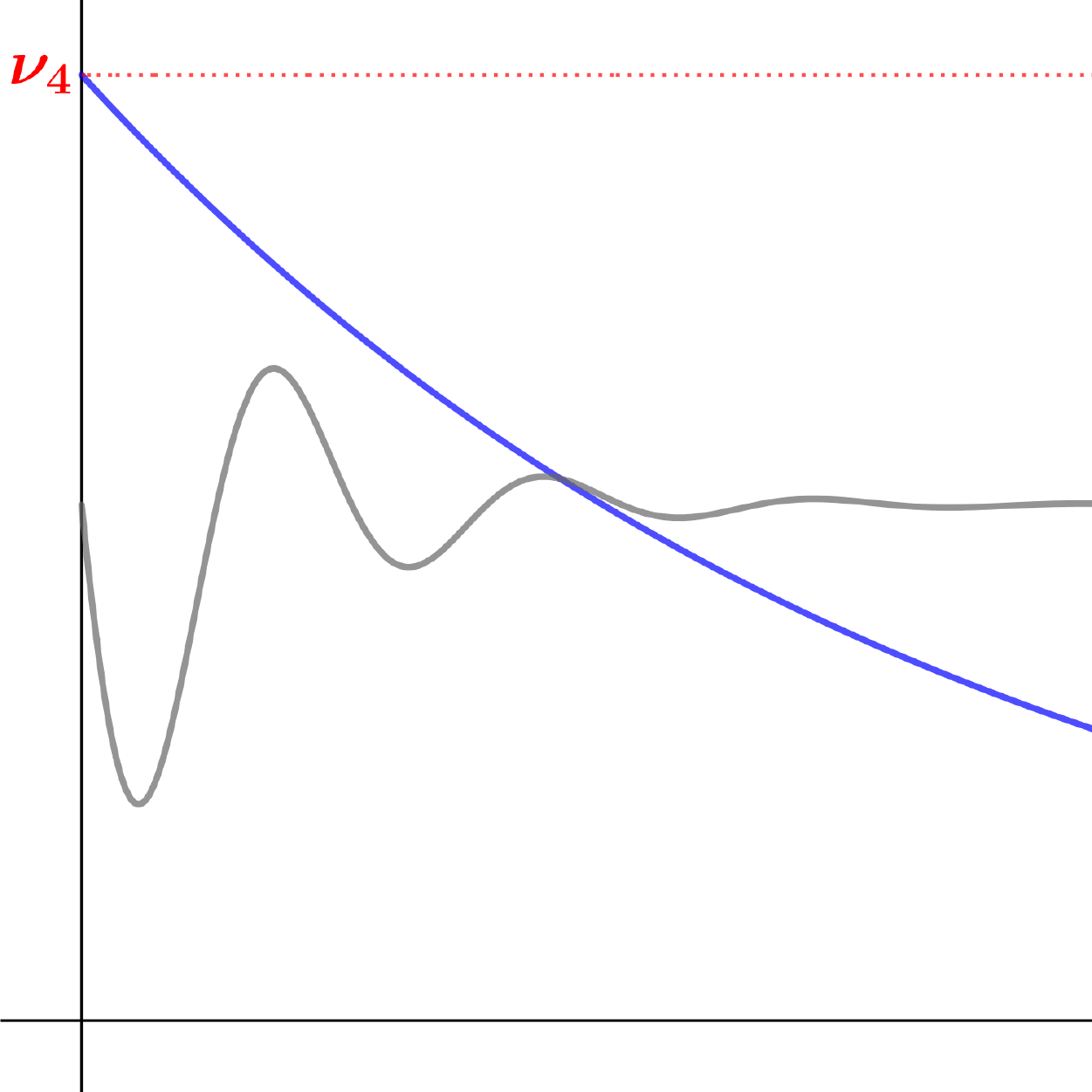} 
	\caption{Graphs of $a_2$ and $\nu c_1^\pm$ (in blue) for different choices of $\nu$}\label{fig:graph_a2}
\end{figure}

\begin{multicols}{2}
	\begin{minipage}{9cm}

In that case, the bifurcation from zero is a subcritical pitchfork bifurcation and three other saddle-node bifurcations occur, two supercritical and one subcritical. The bifurcation curve looks like this:
\end{minipage}

\begin{minipage}{6cm}
\begin{flushright}
	%	\begin{figure}[h!]

%		\centering
		\includegraphics[scale=0.2]{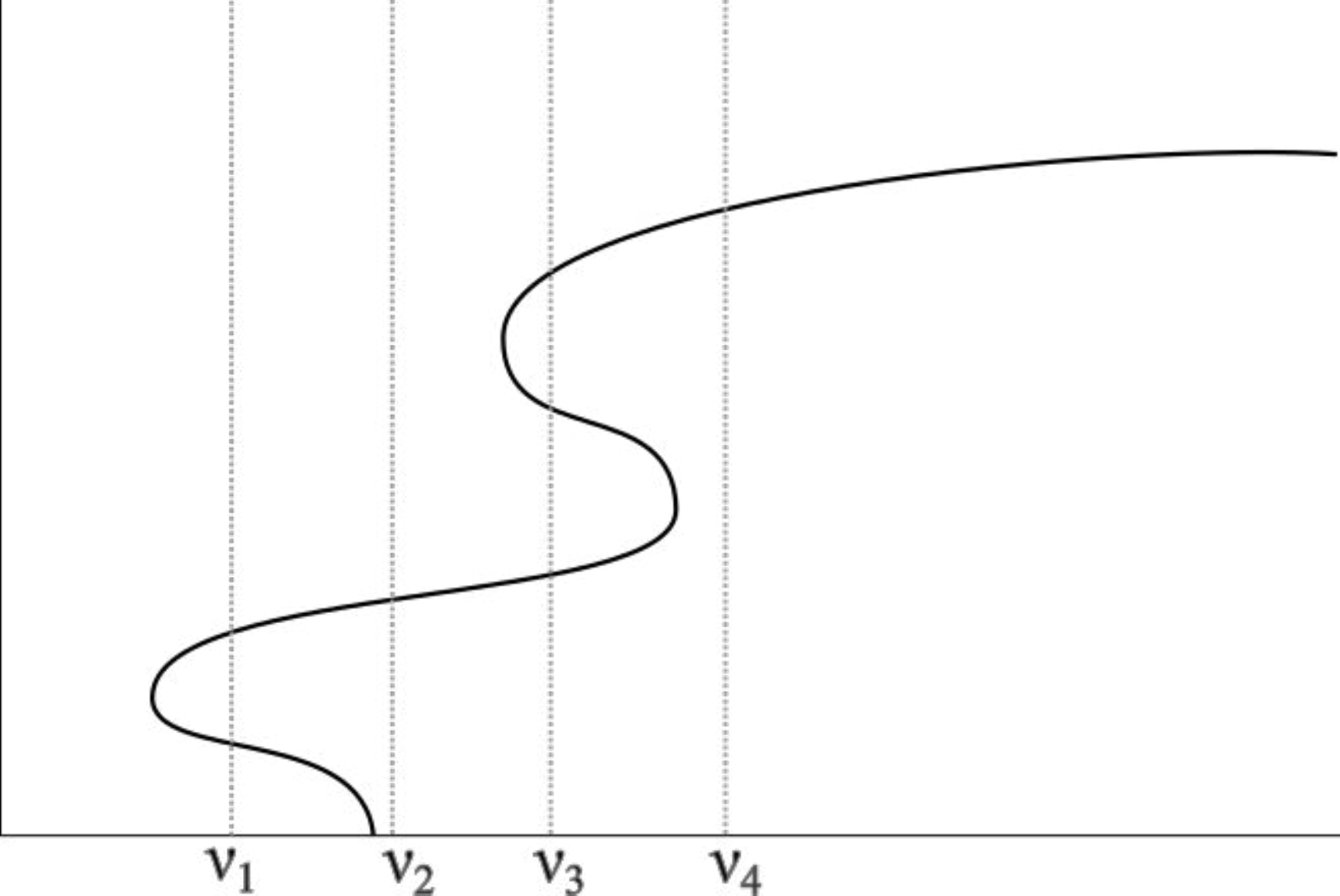}
%		\c%aption{Bifurcation curve. Intersections between the graphs of $a_2$ and $\nu c_j^\pm$ as a function of $\nu$}
%	\end{figure}
\end{flushright}
\end{minipage}
\end{multicols}

\end{example}

\begin{example} Consider the function given by $a=a_3$, with graph pictured in gray, as in Figure \ref{fig:a3}.
	
%	\begin{figure}[htbp]
%	\centering
%	\noindent
%	\includegraphics[scale=0.2]{C1.pdf} %\hspace{0.4cm}
%	\includegraphics[scale=0.2]{C2.pdf} %\hspace{0.4cm}	
%	\includegraphics[scale=0.2]{C3.pdf} %\hspace{0.5cm}
%	\includegraphics[scale=0.2]{C4.pdf} %\hspace{0.4cm}	
%	\includegraphics[scale=0.2]{C5.pdf} %\hspace{0.4cm}
%	\includegraphics[scale=0.2]{C6.pdf}
%	\caption{Graphs of $a_3$, $\nu c_1^\pm$ and $\nu c_2^\pm$ for different choices of $\nu$}
%	\label{fig:a3}
%\end{figure}

	\begin{figure}[htbp]
		\centering
\noindent
		\includegraphics[scale=0.26]{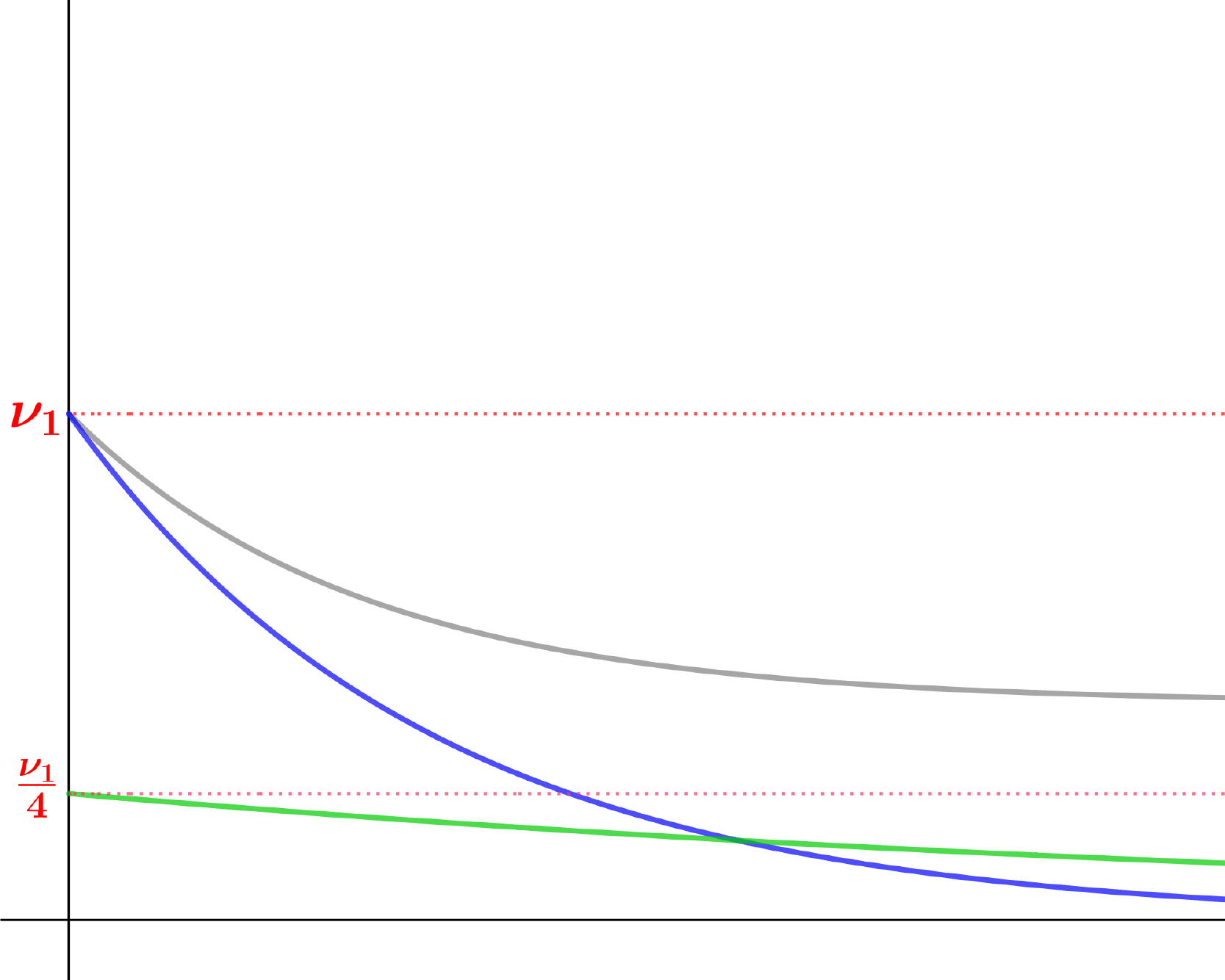} \hspace{0.4cm}
		\includegraphics[scale=0.26]{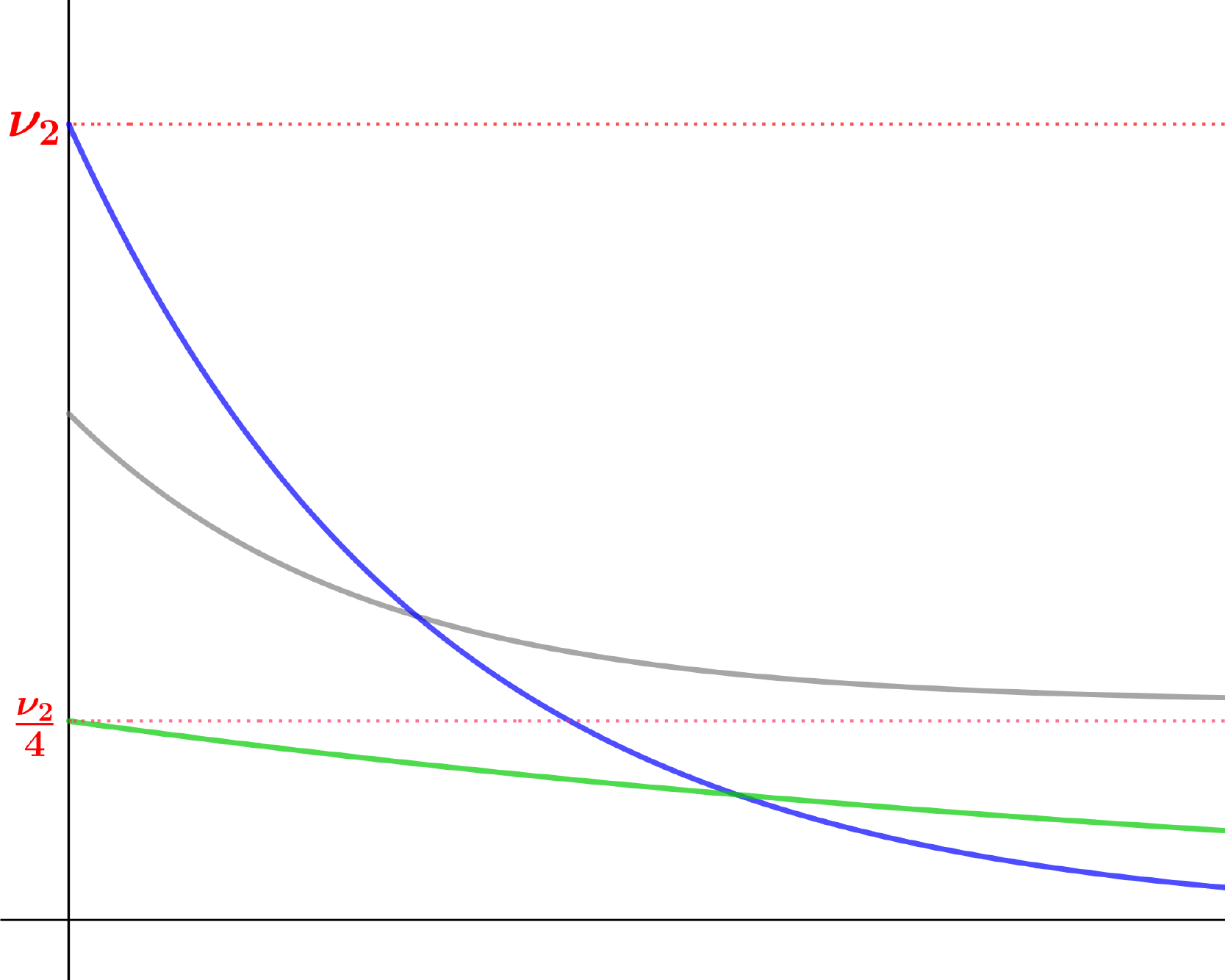} \hspace{0.4cm}	
		\includegraphics[scale=0.26]{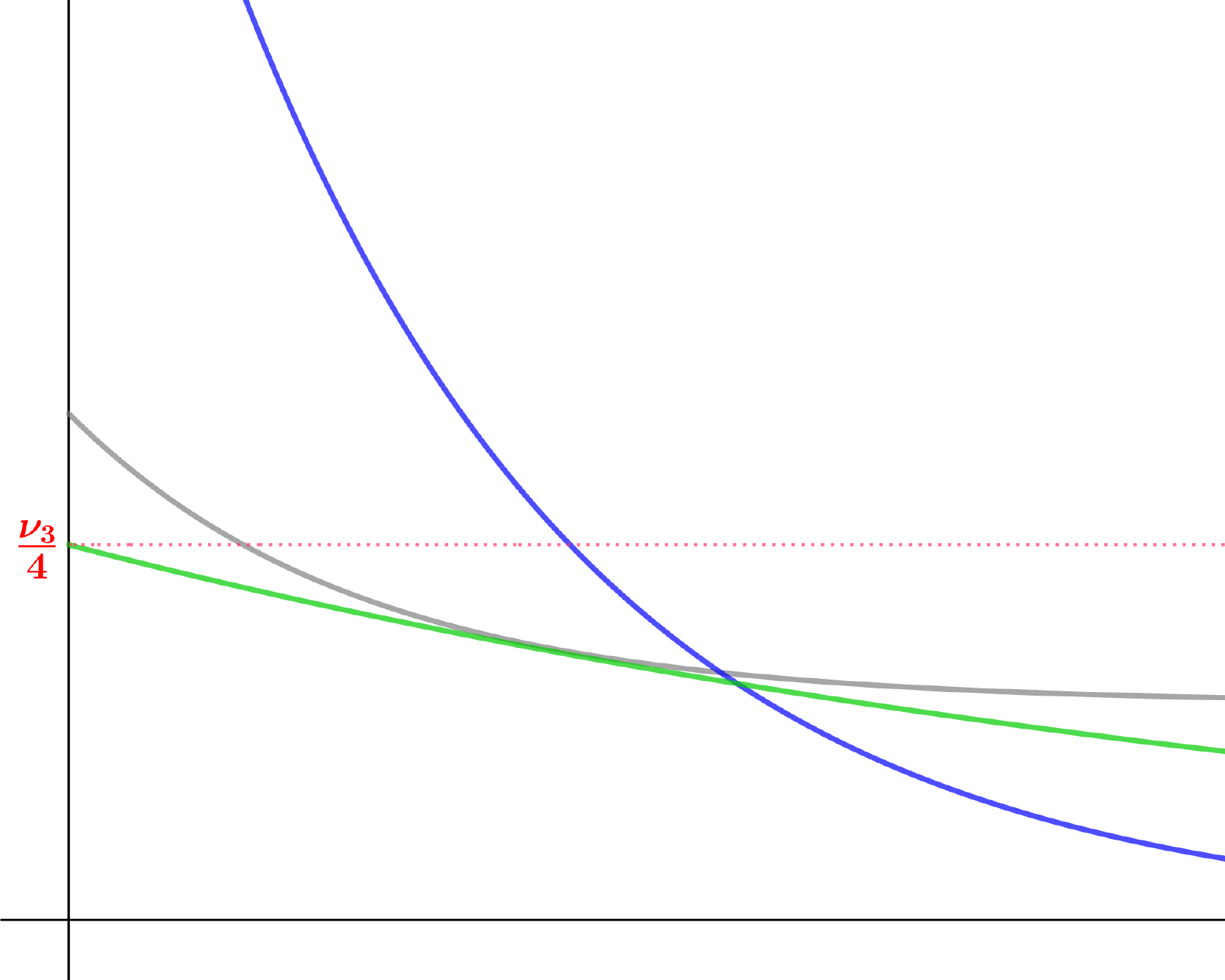} %\hspace{0.5cm}

\vspace{0.5cm}

		\includegraphics[scale=0.26]{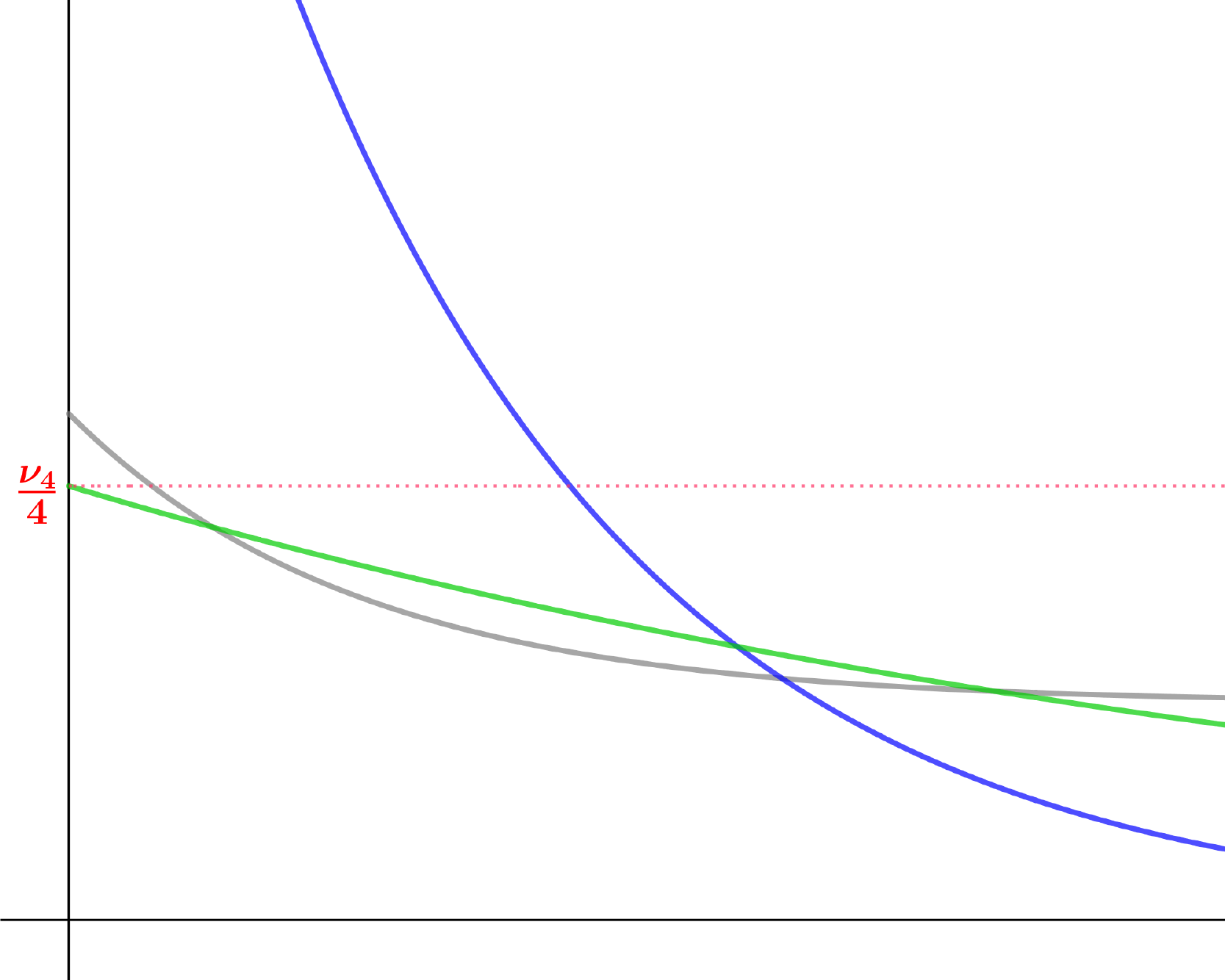} \hspace{0.4cm}	
		\includegraphics[scale=0.26]{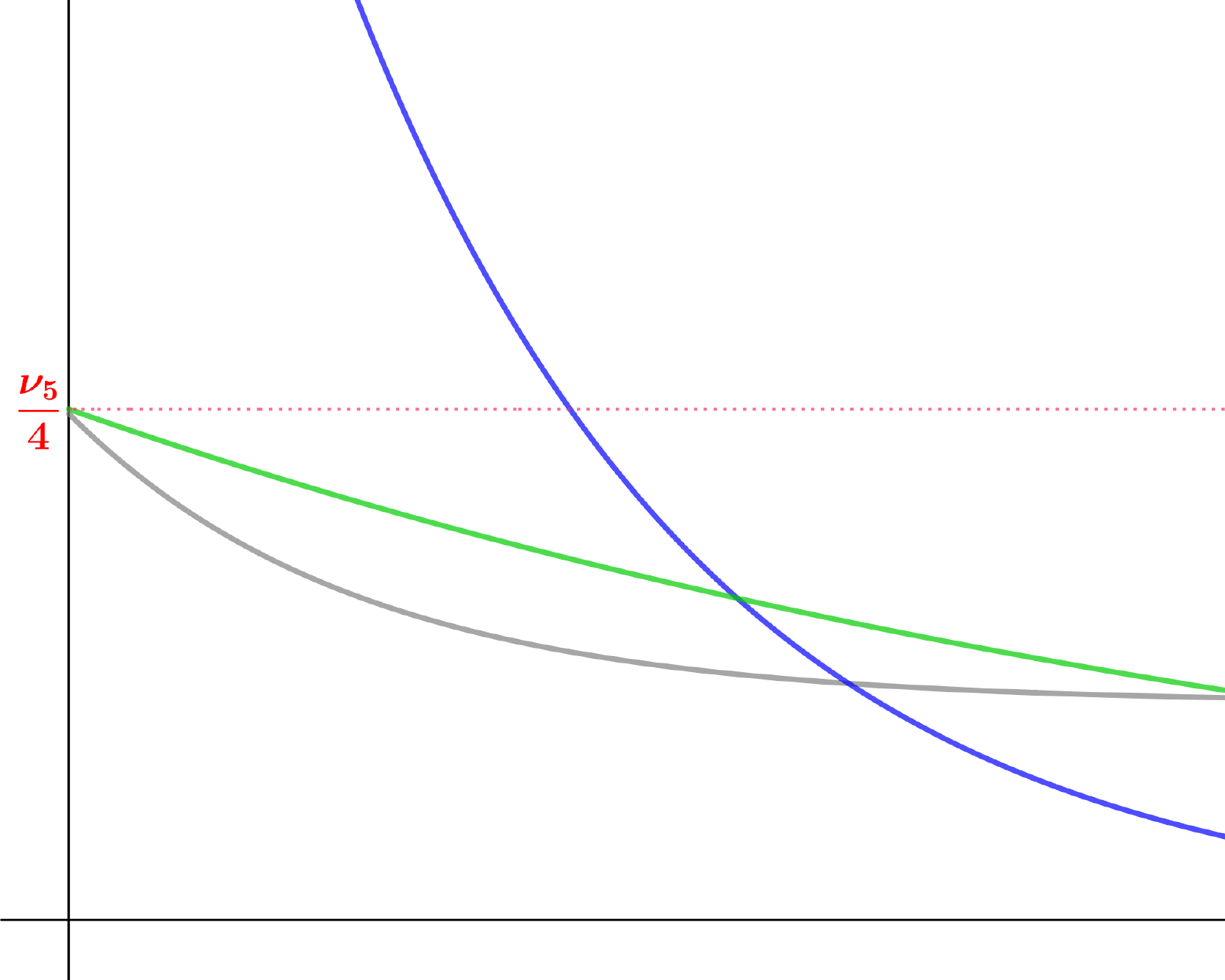} \hspace{0.4cm}
		\includegraphics[scale=0.26]{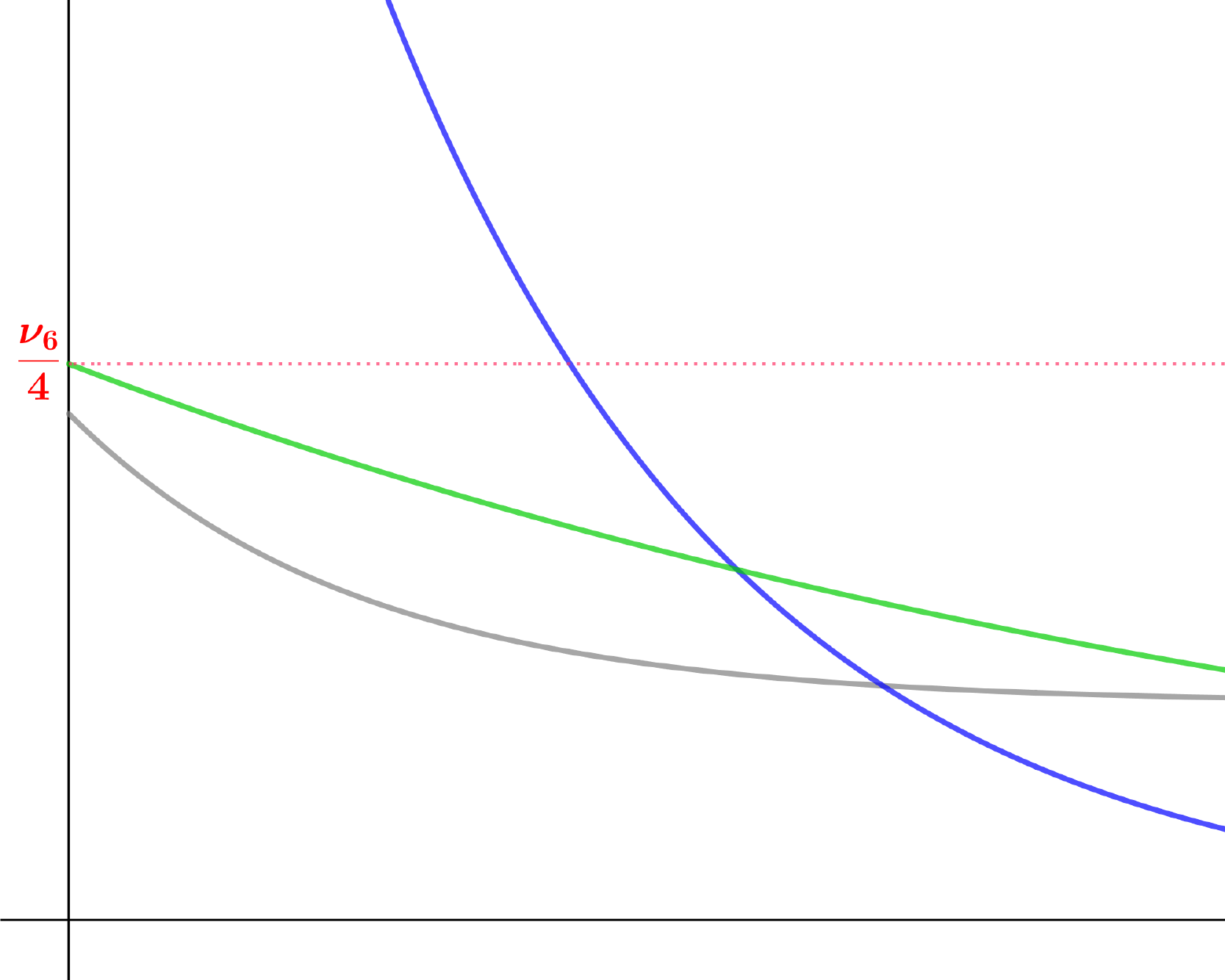}
		\caption{Graphs of $a_3$ (in gray), $\nu c_1^\pm$ (in blue) and $\nu c_2^\pm$ (in green) for different choices of $\nu$}
		 \label{fig:a3}
	\end{figure}

\begin{multicols}{2}
\begin{minipage}{9 cm}

The first bifurcation from zero  is a supercritical pitchfork bifurcation and the second bifurcation from zero is a supercritical saddle-node bifurcation. 

In this case, the diagram representing the two bifurcations from zero is similar to the figure:

\end{minipage}

\begin{minipage}{6 cm}
\begin{flushright}
%	\begin{figure}[h!]
	%		\centering
	\includegraphics[scale=0.2]{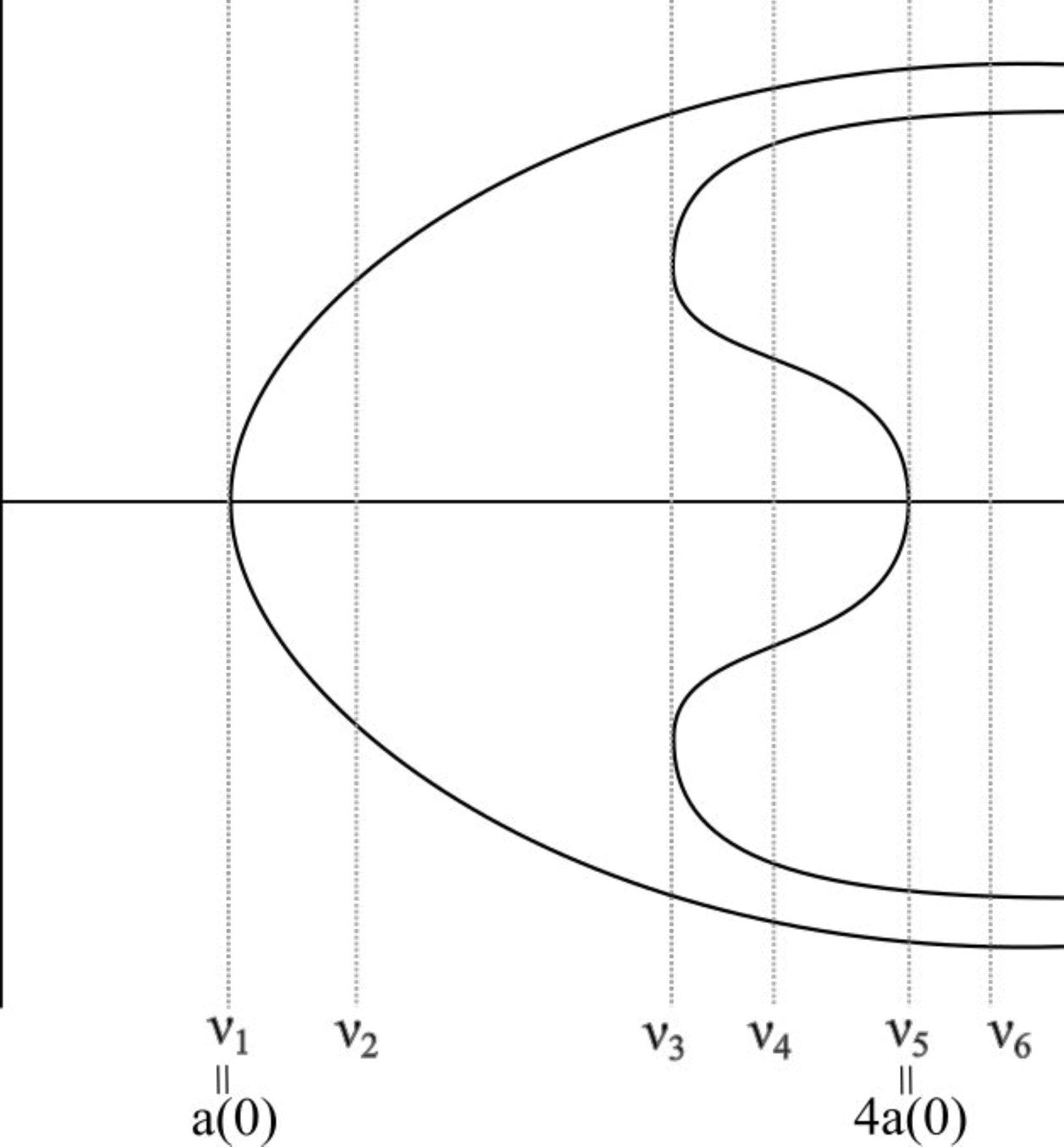}
	%	\caption{Bifurcation curve. Intersections between the graphs of $a_3$ and $\nu c_j^{\pm}$ as a function of $\nu$}	
	%	\end{figure}
\end{flushright}	
\end{minipage}	
\end{multicols}	

Suppose that $\nu_3 \in (\nu_1,\nu_5)$ is the moment for which the saddle-node bifurcation of the equilibria that change sign one time in $(0,\pi)$ appears. In this case, if $f$ is odd, a pictorial representation of the global attractor is given in Figure \ref{attrFig}.	
	
\begin{figure} [h]
		\centering
\includegraphics[scale=0.24]{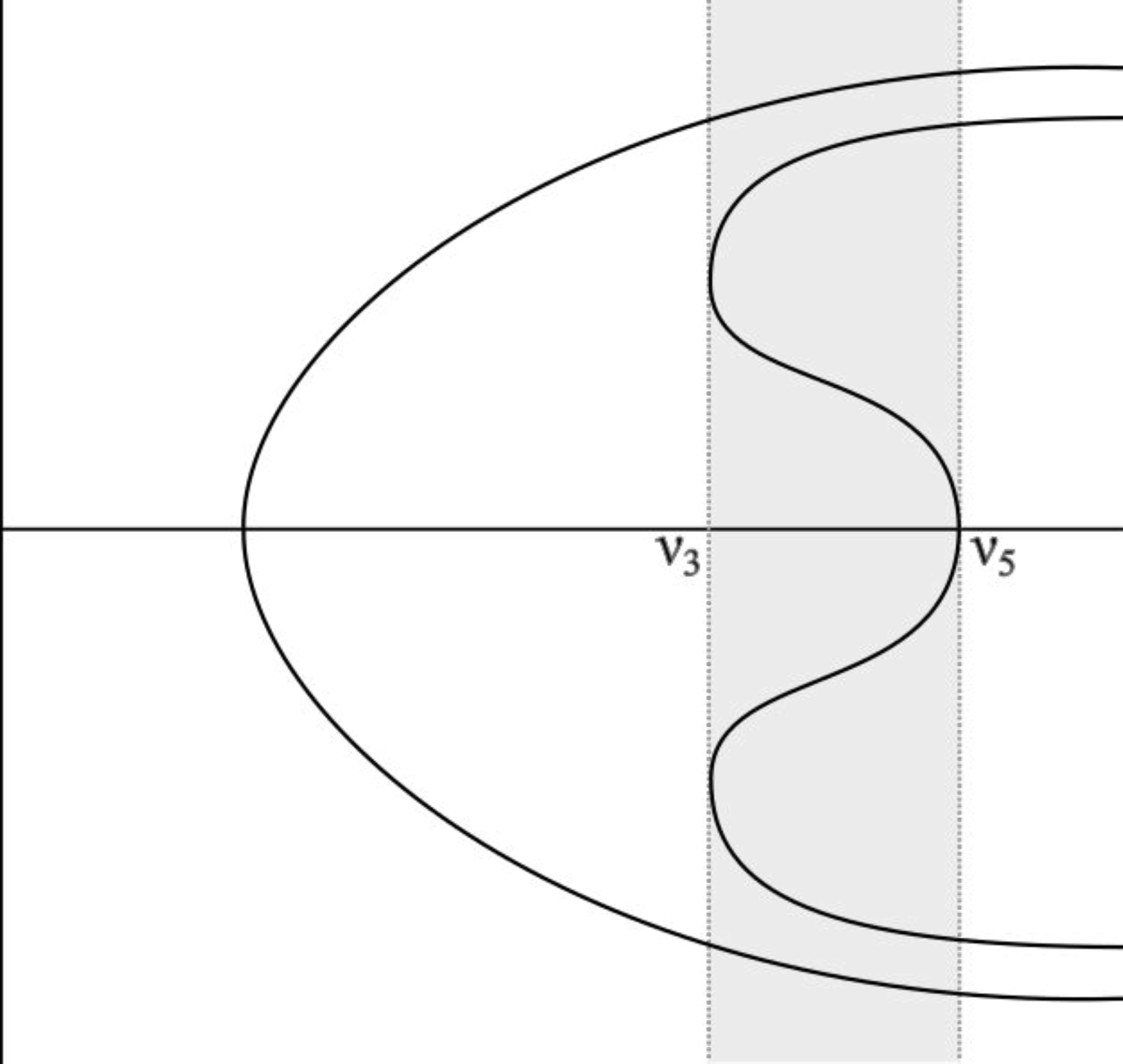}
\hspace{0.5cm}
\includegraphics[scale=0.55]{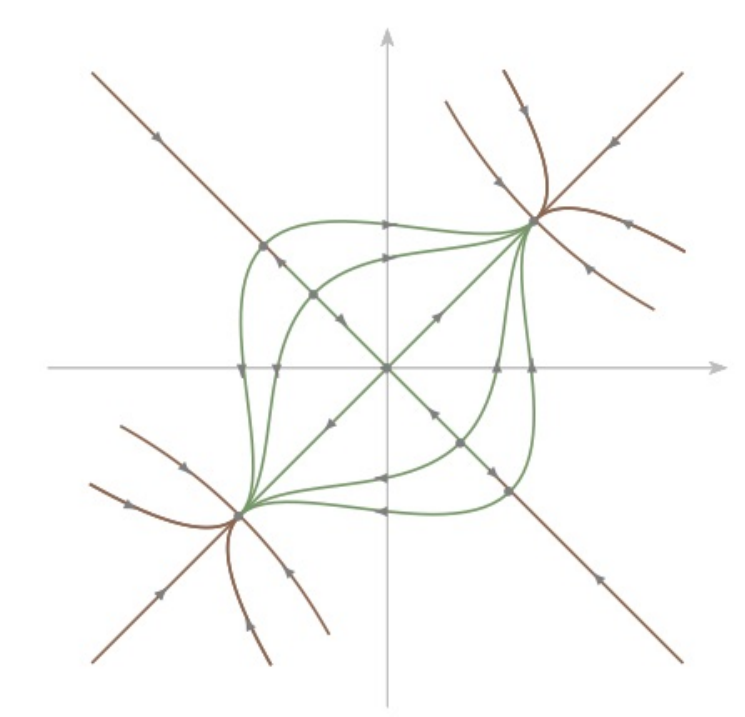} 
\caption{Expected structure of the attractor, when $\nu \in (\nu_3,\nu_5)$.}
\label{attrFig}
\end{figure}

For $\nu\in (\nu_3,\nu_5)$, it is also expected that the two more unstable equilibria collapses at $0$ as $\nu$ approaches $4a(0)$.
\end{example}

\section*{Acknowledgments}

This work was carried out while the third author (EMM) visited the Centro de Investigaci\'on Operativa, UMH de Elche. During this period, she had the opportunity to visit the Universidad Complutense de Madrid. She wishes to express her gratitude to the people from CIO and the UCM for the warm reception and kindness.

%\vfill\eject

\bibliography{biblio}
\bibliographystyle{plain}

\end{document}